\documentclass[10pt,a4wide]{article}

\usepackage{amsfonts}
\usepackage{amsmath}
\usepackage{amsthm}
\usepackage{amssymb}
\usepackage{oldgerm}
\usepackage[]{graphicx}
\usepackage{hyperref}

\newtheorem{theorem}{Theorem}[section]

\newtheorem{proposition}[theorem]{Proposition}
\newtheorem{remark}[theorem]{Remark}
\newtheorem{lemma}[theorem]{Lemma}

\newtheorem{definition}[theorem]{Definition}

\DeclareMathOperator{\supp}{supp}
\DeclareMathOperator{\card}{card}
\DeclareMathOperator{\diam}{diam}

\begin{document}

\title{Asymptotic optimality of scalar Gersho quantizers}
\author{Wolfgang Kreitmeier\footnote{The
author was supported through a grant from the German Research Foundation
(DFG).}}
\date{}
\maketitle

\begin{abstract}
In his famous paper \cite{Ge79} Gersho stressed that the codecells of
optimal quantizers asymptotically make an equal contribution to the distortion of
the quantizer. Motivated by this fact, we investigate in this paper  
quantizers in the scalar case, where each codecell contributes with exactly the same portion to the 
quantization error. We show that such quantizers of Gersho type - or Gersho quantizers for short - 
exist for non-atomic scalar distributions. 
As a main result we prove that Gersho quantizers are asymptotically optimal.   
\end{abstract}

~
 
\pagestyle{myheadings}
\markboth{W. Kreitmeier}{Asymptotic optimality of Gersho quantizers}

\section{Introduction and notation}

Let $\mu$ be a Borel probability measure on $\mathbb{R}^{d}$. A Borel-measurable mapping 
$q : \mathbb{R}^{d} \to \mathbb{R}^{d}$ is called  
quantizer if $\card ( q( \mathbb{R}^{d} ) ) \leq \card( \mathbb{N} )$, where $\card$ denotes cardinality.
For any quantizer $q$, norm $\| \cdot \|$ and norm exponent $r>0$ we define the distortion or quantization error 
\[
D(\mu , q, r) = \int \| x - q(x) \|^{r} d \mu (x) = 
\sum_{a \in q(\mathbb{R}^{d})} \int_{q^{-1}(a)} \| x - a \|^{r} d \mu (x),
\]
where $q^{-1}(a)$ is called codecell of the codepoint $a \in q(\mathbb{R}^{d})$.
The set $q(\mathbb{R}^{d})$ is called codebook.
The quantization error can be interpreted as a measure for the distance between $\mu$ and the image  
$\mu \circ q^{-1}$ of $\mu$ under $q$. 
Indeed, if $\mathbb{R}^{d}$ is equipped with the Euclidean norm and if 
$\mu$ is vanishing on 
continuously differentiable $(d-1)$-dimensional submanifolds of $\mathbb{R}^{d}$,
then the quantization error 
equals the $L_{r}-$Wasserstein distance (see e.g. \cite[Theorem 2.6]{Kr11}).
Let us denote the set of all quantizers on $\mathbb{R}^{d}$ by $\mathcal{Q}$.
For $n \in \mathbb{N}$ we define the optimal $n-$th quantization error for $\mu$ of order $r$ as
\[
D_{n, r}(\mu ) = \inf \{ D(\mu , q, r) : q \in \mathcal{Q} \text{ and } \card ( q( \mathbb{R}^{d} ) ) \leq n \}.
\]
The problem of optimal quantization is to find an optimal $n-$level quantizer for $\mu$, i.e. a 
$q \in \mathcal{Q}$ with $D(\mu , q, r) = D_{n, r}(\mu )$. 
Although an optimal quantizer exists under
weak assumptions on $\mu$ (cf. \cite[Theorem 4.12]{GrLu00}), the determination of the optimal 
quantizers has been achieved so far only for a few distributions 
(see e.g. \cite[p.69ff]{GrLu00}, \cite{FoPa02}, \cite{KeZh07}, \cite{Kr08}).
As another difficulty, often more than one optimal quantizer exists. This phenomenon of non-uniqueness can 
happen for distributions which are absolutely continuous with respect to the Lebesgue measure 
(cf. \cite{AbWi81} resp. \cite[Example 5.2]{GrLu00}) but also for distributions which are singular 
\cite{GrLus97, Kr08}.

Due to these difficulties and also in view of aspects 
in applications
(see \cite{GrNe98, GeGr92}
for an excellent overview), one is more interested in asymptotics
of the optimal quantization error for large quantization levels $n$.
For distributions which have a non-vanishing part that is absolutely continuous 
with respect to the Lebesgue measure and which are satisfying a certain moment condition, 
these asymptotics are well-known \cite{Za63, BuWi82, GrLu00}.
Now $n$-level quantizers are of interest, which induce the same error asymptotics as the optimal ones
if $n$ tends to infinity. Such a sequence of quantizers is called asymptotically optimal.

Delattre et al. \cite{DeFoPa04} have shown for a large class of scalar distributions, that for 
any sequence of asymptotically optimal quantizers, which in addition satisfy a condition of stationarity, 
the codecells contribute asymptotically with equal portion to the distortion of the quantizer.
This behavior was first mentioned by Panter and Dite \cite{PaDi51} for optimal scalar quantizers under 
high rates. T\'{o}th \cite{To59} and Gersho \cite{Ge79} conjectured this asymptotical behavior of optimal
quantizers also for higher-dimensional distributions. 
Now one can question if it is possible to construct 
asymptotically optimal quantizers by using
this property of asymptotically equal moments on each codecell.
Apart from \cite{DeFoPa04}, the formulation of this uniformity is quite vague
or only conjectured. To start with a thorough mathematical formulation
we consider quantizers where each codecell contributes with exactly the same portion to the 
quantization error. 
In the following section
we will show that such quantizers of Gersho type - or Gersho quantizers for short - 
exist for non-atomic scalar distributions and all levels $n$.
Moreover, we will show that Gersho quantizers are even unique for all levels $n$
if $\mu$ has an interval as support.
 
Then, one can ask if such Gersho quantizers are asymptotically optimal. 
As a main result of this paper we will give in section three
a positive answer to this question for a 
large class of scalar probability distributions (cf. Theorem \ref{main_theo}). 
In the last section we provide concluding remarks which are mainly consisting of several 
remaining open questions. 

For the rest of this paper let us assume that $\mu$ is 
one-dimensional and non-atomic, i.e. $\mu(\{ x \}) = 0$ 
for every $x \in \mathbb{R}$.
Moreover, we assume throughout the paper that $\mu$ has a finite $r-$th moment.
For any Borel-measurable $A \subset \mathbb{R}$ with $\mu (A) > 0$ we denote by
$\mu ( \cdot | A )$ the conditional probability of $\mu$ with respect to $A$.
We denote by $\mathcal{C}(n,\mu, r)$ the set of all $n-$optimal quantizers for $\mu$ of order $r$.

\begin{definition}
We call $q \in \mathcal{Q}$ an $n-$level Gersho-quantizer of order $r$ for $\mu$ if
\begin{itemize}
\item[(G1)]
$\card( q(\mathbb{R}) ) = n$,
\item[(G2)]
$q^{-1}(a)$ is an interval for every $a \in q( \mathbb{R} )$,
\item[(G3)]
every $a \in q( \mathbb{R} )$ is optimal for its own codecell, i.e.
the mapping $\mathbb{R} \ni x \to a$ is an element 
of $\mathcal{C}(1,\mu ( \cdot | q^{-1}(a) ), r)$ and
\item[(G4)]
every codecell of $q$ contributes an equal amount to the overall distortion, i.e. 
$\int_{q^{-1}(a)} |x-a|^{r} d \mu (x) = \frac{1}{n} D(\mu , q , r)$ for every $a \in q( \mathbb{R} )$.
\end{itemize}
Let us denote by $\mathcal{G}(n,\mu, r)$ the set of all 
$n-$level Gersho-quantizers of order $r$ for $\mu$.
\end{definition}

\begin{remark}
\label{ref_mue_gz}
Because $\mu$ is non-atomic, we know that
$D(\mu , q , r) > 0$ for every $q \in \mathcal{G}(n,\mu, r)$.
Consequently, property (G4) implies that $\mu(q^{-1}(a)) > 0$ 
for every $a \in q( \mathbb{R} )$, i.e. all codecells have non-vanishing $\mu-$mass and, 
therefore, $\mu ( \cdot | q^{-1}(a) )$ in (G3) is well defined. 
\end{remark}

\begin{remark}
\label{rem_voro}
Let $q \in \mathcal{C}(n,\mu, r)$ be an $n-$optimal quantizer.
If we consider the codebook $q(\mathbb{R})$ and a Voronoi partition $\{ A_{a} : a \in  q(\mathbb{R}) \}$
of $\mathbb{R}$ with respect to $q(\mathbb{R})$, i.e.
\[
A_{a} \subset \{ x \in \mathbb{R} : | x - a |  = \min_{b \in q(\mathbb{R})} | x - b | \} 
\quad \mu - \text{a.s. for every } a \in q(\mathbb{R}) ,
\]
then the quantizer $q: \mathbb{R} \to \mathbb{R}$ with $q(x)=a$ if $x \in A_{\alpha}$ 
is also an $n-$optimal one
(cf. \cite[Lemma 3.1]{GrLu00}). Due to this fact, the determination of optimal quantizers is reduced to the 
determination of an optimal $n-$codebook, where the codecells are specified by an arbitrary Voronoi 
partition related to this set. Let us call such quantizers Voronoi quantizers.
Insofar we can and will assume w.l.o.g. that every $n-$optimal quantizer is a Voronoi quantizer and
satisfies (G1), (G2) and (G3), see e.g.
\cite[Lemma 3.1]{GrLu00} and \cite[Theorem 4.1]{GrLu00}. 
Gersho quantizers and optimal quantizers even coincide for some distributions.
E.g. if $\mu$ is the uniform distribution on $[0,1]$, then it follows straightforward from property (G4)
and \cite[Example 5.5]{GrLu00} that $\mathcal{G}(n,\mu, r) = \mathcal{C}(n,\mu, r)$
for every $n \in \mathbb{N}$ and $r > 1$.
In general, this is not the case. Even 
\[
\mathcal{G}(n,\mu, r) \cap \mathcal{C}(n,\mu, r) = \emptyset
\]
and
\[
D_{n,r}(\mu ) < \inf \{ D(\mu ,q, r ) : q \in \mathcal{G}(n,\mu, r) \} 
\]
for certain distributions $\mu$, quantization levels $n$ and norm exponent $r$
is possible, see e.g. $\mu$ as from \cite{AbWi81} resp. \cite[Example 5.2]{GrLu00} with $n=2$ and $r=2$.
It remains open to characterize the Gersho quantizers which are Voronoi quantizers.
\end{remark}

Asymptotically optimal quantizers in the scalar case can be constructed by companding
techniques \cite{Li91}. Alternatively, quantizers - which are unique and optimal
for strongly unimodal scalar distributions (cf. \cite{Tr84}) - can be numerically determined by 
the famous Lloyd algorithm \cite{Ll82}.
To get an overview of the historical development and the numerous modifications of the
Lloyd algorithm the user is referred to \cite{GrNe98}.
One drawback of all these Lloyd methods so far is that at any level $n$ the numerical 
calculation of the quantizer starts from scratch, i.e. the calculation results from lower
quantization levels can not be used to reduce calculation complexity for level $n$.
Moreover, the limit point of the Lloyd algorithm is not always necessarily an 
optimal one.

Clearly, the main goal of this paper is a theoretical one. In the following sections 
we investigate existence, uniqueness and asymptotic optimality of Gersho quantizers. 
Nevertheless, let us mention here that Gersho quantizers are also
of practical relevance in constructing asymptotically optimal quantizers. 

Using the results of this paper, we can determine (for a large class of probabilities) numerically 
the sequence of asymptotically optimal Gersho quantizers where the calculations from lower quantization levels 
are incorporated. 
In more detail, let $k \in \mathbb{N}$ and $n=2^{k}$. 
For $k=1$ Proposition \ref{ref_prop_nempt} ensures the existence of a unique Gersho quantizer. 
We can determine this quantizer numerically e.g. by a bisection algorithm (cf. Remark \ref{rem_bisec}). 
If $k>1$, then the uniqueness of Gersho quantizers (cf. Proposition \ref{ref_prop_nempt}) enables us to determine
the Gersho quantizer for level $n=2^{k}$ by dividing all codecells of the quantizer for level $2^{k-1}$
into two cells having the some moment, such that (G4) is satisfied. This dispartment of the codecells
is also done numerically by application of a bisection algorithm.
Theorem \ref{main_theo} yields the asymptotic optimality of the quantizer sequence.

As a main advantage in contrast to any Lloyd algorithm we always determine a unique solution. 
Moreover, this algorithm is also amenable to a parallel implementation.
Of course, we can determine numerically also the Gersho quantizers 
for level $n=n_{0}2^{k}$ with fixed $n_{0} \geq 2$ and
$k \in \mathbb{N}$ or even any arbitrary level $n$. But in this last general case we 
are unable to use the calculation results from lower quantization levels for the actual one.

\section{Existence and uniqueness of Gersho-quantizers}

As already mentioned in the introduction, optimal quantizers exist under
weak assumptions on $\mu$, see \cite{AbWi82}, \cite{Po82} and \cite[Theorem 4.12]{GrLu00}. 
In this section we will show by Proposition \ref{ref_prop_nempt}, 
that Gersho quantizers also always exist for
non-atomic distributions. 
Before studying uniqueness we need to make some limitations for Gersho quantizers. 
Let $\mathcal{F}(n, \mu , r)$ be the set of all scalar quantizers where all codecells of such a 
quantizer have non-vanishing $\mu-$mass and (G1), (G2) and (G3) are satisfied. 
By definition and in view of Remark \ref{rem_voro} resp. Remark \ref{ref_mue_gz} we know that 
\[
\mathcal{G}(n, \mu , r) \cup \mathcal{C}(n, \mu , r) \subset \mathcal{F}(n, \mu , r) .
\]
Let $q \in \mathcal{F}(n, \mu , r)$ and let $I_{1}(q),\dots , I_{n}(q)$ be the
codecells of $q$ in increasing order.
Let $a_{i}=\inf (I_{i}(q))$ and $b_{i}=\sup (I_{i}(q))$ for every $i \in \{1,..,n\}$.
Clearly, 
\begin{equation}
\label{aibi}
a_{1}=-\infty ,  b_{n}=\infty  \text{ and } a_{i}=b_{i-1} \text{ for every } i \in \{2,..,n\} .
\end{equation}
Because $\mu$ is non-atomic let us assume w.l.o.g. throughout this paper that 
\[
I_{1}(q) = (a_{1} , b_{1}), \qquad I_{i}(q)=[a_{i}, b_{i}) \qquad \text{for every  } i \in \{2,..,n\} . 
\]
Thus, the codecells of any quantizer $q \in \mathcal{F}(n, \mu , r)$
are completely characterized by their boundary points.
With these conventions we will also investigate in this section the uniqueness of 
Gersho quantizers. 
Important for our subsequent argumentation is the following well-known result.

\begin{proposition}
\emph{\textbf{(Uniqueness of $1-$optimal quantizers)}}
\label{unique_1_quant}
If $r>1$, then always exactly one $1-$optimal quantizer for $\mu$ exists, i.e. 
$\card ( \mathcal{C}(1, \mu , r) ) = 1$.
\end{proposition}
\begin{proof}
The assertion follows from the strict convexity of the mapping 
\[
\mathbb{R} \ni a \to \int |x-a|^{r} d \mu (x) \in (0, \infty ).
\]
For a complete proof see e.g. \cite[Theorem 2.4]{GrLu00}.
\end{proof}

In the following, Proposition \ref{lemm_equal} shows that every 
$n$-level Gersho-quantizer induces the same quantization error. Moreover, 
Proposition \ref{ref_prop_nempt} states that exactly one $n$-level 
Gersho quantizer exists, i.e.  
$\card ( \mathcal{G}(n,\mu, r) ) =1$
if the support of $\mu$ is an (possibly unbounded) interval.

For any $-\infty \leq a < b \leq \infty$ we define $I_{a,b}=(a,b)$.
If $\mu ( I_{a,b} ) > 0$ and $r>1$, then Proposition \ref{unique_1_quant} 
implies that always exactly one $1$-optimal quantizer for $\mu ( \cdot | I_{a,b} )$
of order $r$ exists, i.e. $\mathcal{C}(1, \mu ( \cdot | I_{a,b} ) , r)$ 
is not empty and consists of exactly one
element. We denote by $c_{a,b} \in \mathbb{R}$ the point which 
represents this unique optimal
quantizer of $\mu ( \cdot | I_{a,b} )$, i.e. 
\[
\int | x - c_{a,b} |^{r} d \mu ( \cdot | I_{a,b} ) (x) = D_{1,r}(\mu ( \cdot | I_{a,b} )) .
\]
In view of \cite[Lemma 2.6(a)]{GrLu00} we additionally know that $c_{a,b} \in [a,b]$.
If we want to stress the dependency of $c_{a,b}$ on $\mu$ we write $c_{a,b}(\mu )$.
Let us denote by $\supp ( \mu )$ the support of $\mu$.
For any set $I \subset \mathbb{R}$ we denote by $1_{I}$ the indicator function of $I$.
Now let us consider the $r-$th moment $\int_{I_{a,b}} | x - c_{a,b} |^{r} d \mu (x)$ 
on $I_{a,b}$. It seems to be obvious that this moment increases with $b$.
Let us prove this important result. 

\begin{lemma}
\label{lemm_ex_01}
Let $-\infty \leq a < b < \infty$ and 
\[
J=( a-b , \infty ) \cap ( \inf ( \supp ( \mu ) ) , \sup ( \supp ( \mu ) ) ) .
\]
If $\mu ( I_{a,b} ) > 0$ and $r>1$, then the mapping
\begin{equation}
\label{def_wab}
J \ni z \to W_{a,b, \mu }(z) = \int_{I_{a,b+z}} | x - c_{a,b+z}|^{r} d \mu (x) 
\end{equation}
is continuous and increasing. 
Moreover, for every $x,y \in J$ with $x < y$
\[
W_{a,b, \mu }(x) < W_{a,b, \mu }(y) \text{ if and only if } \mu ( [x+b,y+b] ) > 0.
\]
\end{lemma}
\begin{proof}
Note that $\mu ( I_{a,b+z} ) > 0$ for every $z \in J$. \smallskip \\
1. First we will show that $J \ni z \to c_{a,b+z}$ is continuous. \smallskip \\
Let $z \in J$.
We proceed indirectly and assume that an $M>0$ and a sequence $(\varepsilon_{n})_{n \in \mathbb{N}}$ exist with
$z+\varepsilon_{n} \in J$ for every $n \in \mathbb{N}$,
$\varepsilon_{n} \to 0$ as $n \to \infty$ and 
\begin{equation}
\label{dist_point}
|c_{a,b+z+\varepsilon_{n}} - c_{a,b+z}| \geq M
\end{equation}
for every $n \in \mathbb{N}$.
Since the mapping $\mathbb{R} \ni w \to \int_{I_{a,b+z}} | x - w |^{r} d \mu (z)$ is strictly convex
(see proof of Proposition \ref{unique_1_quant}) we deduce that
\begin{eqnarray}
C_{1}(M) & := &
\max \{ \int_{I_{a,b+z}} | x - w |^{r} d \mu (x) : w \in \{ c_{a, b+z} - M, c_{a, b+z} + M \}  \}  
\nonumber \\
&=& \inf \{ \int_{I_{a,b+z}} | x - w |^{r} d \mu (x) : |w-c_{a, b+z}| \geq M \} .
\label{def_c1m}
\end{eqnarray}
Let 
\[
C_{2}(M) = C_{1}(M) - \int_{I_{a,b+z}} | x - c_{a,b+z} |^{r} d \mu (x) > 0.
\]
Now choose $n_{0} \in \mathbb{N}$ such that 
\begin{equation}
\label{diff_c2m}
\int_{I_{a,b+z+\varepsilon_{n}}} | x - c_{a, b+z} |^{r} d \mu (x) -
\int_{I_{a,b+z}} | x - c_{a, b+z} |^{r} d \mu (x) < C_{2}(M)
\end{equation}
for every $n \geq n_{0}$. Now let $n \geq n_{0}$.
Combining (\ref{dist_point}) and (\ref{def_c1m}) we observe that
\begin{eqnarray*}
\int_{I_{a,b+z+\varepsilon_{n}}} | x - c_{a, b+z+\varepsilon_{n}} |^{r} d \mu (x) 
\geq  \int_{I_{a,b+z}} | x - c_{a, b+z+\varepsilon_{n}} |^{r} d \mu (x) \geq C_{1}(M).
\end{eqnarray*}
Using (\ref{diff_c2m}) we get
\begin{eqnarray*}
&& \int_{I_{a,b+z+\varepsilon_{n}}} | x - c_{a, b+z+\varepsilon_{n}} |^{r} d \mu (x) \\
& \geq &
C_{2}(M) + \int_{I_{a,b+z}} | x - c_{a,b+z} |^{r} d \mu (x) >
\int_{I_{a,b+z+\varepsilon_{n}}} | x - c_{a, b+z} |^{r} d \mu (x), 
\end{eqnarray*}
which contradicts the optimality of $c_{a, b+z+\varepsilon_{n}}$.
Hence, $J \ni z \to c_{a,b+z}$ is continuous.
\smallskip \\
2. Rest of the proof. \smallskip \\
Let $\varepsilon > 0$ and $z \in J$. 
For any $\delta > 0$ we define
\[
C_{1}(\delta ) = \sup \{ | c_{a,b+v} | : v \in J, |v-z| < \delta  \}.
\]
By step 1 we can find a $\delta_{1} > 0$ such that $C_{1}(\delta_{1} ) < \infty$.
For every $v \in J$ with $|z-v| < \delta_{1}$ we have
\[
| x - c_{a,b+v} |^{r} \leq ( 2 \max ( |x|, |c_{a,b+v}| ) )^{r} \leq
2^{r}( | x |^{r} + C_{1}(\delta_{1} )^{r}),
\]
where the right hand side is $\mu$-integrable, because the $r-$th moment of $\mu$
is finite.
Thus a $\delta_{2} \in ( 0, \delta_{1} )$ exists with 
\begin{equation}
\label{upp_bou_mue}
\int_{[b+z-\delta_{2}, b+z+\delta_{2}]} | x - c_{a,b+v} |^{r} d \mu (x) < \frac{\varepsilon}{2}
\end{equation}
for every $v \in J$ with $|z-v|<\delta_{2}$.
For every $\delta_{3} \in ( 0, \delta_{2} )$ we define  
\[
C_{2}(\delta_{3}) = \sup \{ | c_{a,b+z} - c_{a,b+v} | : v \in J , |v - z| < \delta_3  \} < \infty .
\]
Applying step 1 we deduce that
\begin{equation}
\label{conv_to_zero}
C_{2}(\delta) \to 0 \text{ as } \delta \to 0.
\end{equation}
For any $v \in J $ with $|v - z| < \delta_3$ and 
$x \in I_{a,b + z + \delta_{3} }$
we define 
\begin{equation}
\label{def_integr}
f(x,z,v) = | | x - c_{a,b+z}|^{r} -  | x - c_{a,b+v}|^{r} |.
\end{equation}
Clearly,
\begin{equation}
\label{dom_conv_b}
f(x,z,v) \leq ( |x| + |c_{a,b+z}| )^{r}  + ( |x| + |c_{a,b+z}| + C_{2}(\delta_{3}) )^{r}.
\end{equation}
Due to $C_{2}(\delta_{3}) < \infty$ the 
right hand side of (\ref{dom_conv_b}) is $\mu$-integrable.
From (\ref{conv_to_zero}) and (\ref{def_integr}) we deduce that 
$f(x,z,v) \to 0$ as $v \to z$.
Hence, by dominated convergence a $\delta_{4} \in ( 0 , \delta_{3} )$ exists, such that
\begin{equation}
\label{int_upp_b_eps}
\int_{I_{a,b + z }} f(x,z,v) d \mu (x) < \varepsilon / 2 ,
\end{equation}
if $|v - z| < \delta_{4}$.
Now let $v \in J$ such that $|v - z| < \delta_{4}$.
We obtain
\begin{eqnarray*}
&& | W_{a,b, \mu }(z)   - W_{a,b, \mu }(v) | \\
& \leq & | \int_{I_{a,b+\max(v,z)}} | x - c_{a,b+\max(v,z)}|^{r} d \mu (x) \\
&& \qquad \qquad - \int_{I_{a,b+\min(v,z)}} | x - c_{a,b+\max(v,z)}|^{r} d \mu (x) | \\
&+& \quad | \int_{I_{a,b+\min(v,z)}} | x - c_{a,b+\max(v,z)}|^{r} d \mu (x) \\
&& \qquad \qquad  - \int_{I_{a,b+\min(v,z)}} | x - c_{a,b+\min(v,z)}|^{r} d \mu (x) | \\
&\leq & \int_{[b+z-\delta_{4}, b+z+\delta_{4}]} | x - c_{a,b+\max(v,z)} |^{r} d \mu (x)  
+ \int_{I_{a,b+z}} f(x,z,v) d \mu (x). 
\end{eqnarray*}
From (\ref{upp_bou_mue}) and (\ref{int_upp_b_eps}) we deduce that
\[
| W_{a,b, \mu }(z)   - W_{a,b, \mu }(v) | < \varepsilon ,
\]
which yields that the mapping is continuous.
Finally we observe for any 
$x,y \in J$ with $x < y$
that
\begin{eqnarray}
&& W_{a,b, \mu }(y )  - W_{a,b, \mu }(x)  \nonumber \\
& = &
\int_{I_{a,b+y }} | z - c_{a,b+y }|^{r} d \mu (z) 
- \int_{I_{a,b+x}} | z - c_{a,b+x}|^{r} d \mu (z) \nonumber \\
& \geq & \int_{I_{b+x,b+y }} | z - c_{a,b+y }|^{r} d \mu (z) \geq 0,
\label{wabz}
\end{eqnarray}
which implies monotony. Recall that $\mu$ is non-atomic.
Hence, inequality (\ref{wabz}) is strict if and only if $\mu ( [x+b,y+b] )>0$.
\end{proof}

Because we want to show now 
that every $n-$level Gersho quantizer induces the same quantization error, the following
definition makes sense.

\begin{definition}
Let $n \in \mathbb{N}$ and assume that $\mathcal{G}(n,\mu , r) \neq \emptyset$. 
We define
\[
D_{n,r}^{G}(\mu ) = \inf \{ D(\mu , q, r) : q \in \mathcal{G}(n,\mu , r) \}
\]
as the optimal $n-$th Gersho-quantization error for $\mu$ of order $r$.
\end{definition}

Let 
\[
\mathbb{R} \ni x \to T(x)=-x
\]
be the reflection with respect to the origin.

Clearly, the distortion of 
every $n-$optimal quantizer $q \in \mathcal{C}(n,\mu, r)$ equals $D_{n, r}(\mu )$.
Therefore, it is natural to ask if a similar result holds for Gersho
quantizers. The following proposition gives a positive answer.

\begin{proposition}
\emph{\textbf{(Uniqueness of overall distortion)}}
\label{lemm_equal}
~\\
Let $n \in \mathbb{N}$ and $r>1$. Assume that 
an $n-$level Gersho quantizer exists, i.e. $\mathcal{G}(n,\mu , r) \neq \emptyset$.
If $q \in \mathcal{G}(n,\mu , r)$, then the distortion induced by $q$ equals the 
$n-$th Gersho-quantization error, i.e.
\[
D(\mu , q , r) = D_{n,r}^{G}(\mu ) .
\]
\end{proposition}
\begin{proof}
If $n=1$, then the assertion follows immediately from (G3) and Proposition \ref{unique_1_quant}.
Let $n \geq 2$.
Let $I_{1}(q),\dots , I_{n}(q)$ be the codecells of $q$ 
in increasing order of $\mathbb{R}$.
Moreover, let $p \in \mathcal{G}(n,\mu , r)$ and
denote by $I_{1}(p),\dots , I_{n}(p)$ the codecells of $p$ 
in increasing order, too.
We proceed indirectly and assume w.l.o.g. that 
\begin{equation}
\label{contrad}
D(\mu , q , r) < D(\mu , p , r).
\end{equation}
The idea of the proof is as follows. 
Using (\ref{contrad}),(G4) and Lemma \ref{lemm_ex_01} one can show  
that the moment and the right endpoint of the first codecell $I_{1}(q)$ of $q$ is 
strictly smaller than the one of $I_{1}(p)$. By induction this holds 
up to the penultimate codecell of $q$ and $p$. Consequently, the last codecell $I_{n}(q)$ of
$q$ is strictly larger (according to size and moment) than the last one of $p$, which
contradicts (\ref{contrad}).
Now let us elaborate this idea in detail. 
We proceed in two steps.
\smallskip \\
1. We will show that
\[
\sup I_{n-1}(q) < \sup I_{n-1}(p) \quad \text{ and } \mu ( [ \sup I_{n-1}(q), \sup I_{n-1}(p) ] ) > 0.
\]
We proceed by induction and assume first that $n=2$.
According to (\ref{aibi}), (\ref{contrad}) and by (G4) we know that
\begin{eqnarray*}
&& W_{-\infty , 0, \mu }(\sup I_{1}(q) ) \\
&=& \int_{I_{1}(q)} |x - c_{-\infty , \sup I_{1}(q)}|^{r} d \mu (x) <
\int_{I_{1}(p)} |x - c_{-\infty , \sup I_{1}(p)}|^{r} d \mu (x) \\
&=& W_{-\infty , 0, \mu }(\sup I_{1}(p) ). 
\end{eqnarray*}
Applying Lemma \ref{lemm_ex_01} we obtain 
\[
\sup I_{1}(q) < \sup I_{1}(p) \text{ and } \mu ( [ \sup I_{1}(q), \sup I_{1}(p) ] ) > 0.
\]
Now let us assume that $n > 2$, 
\[
\sup I_{n-2}(q) < \sup I_{n-2}(p) \text{ and that } \mu ( [ \sup I_{n-2}(q), \sup I_{n-2}(p) ] ) > 0.
\]
Again from (\ref{contrad}) and (G4) we deduce that
\begin{eqnarray}
&& W_{\inf I_{n-1}(q),0, \mu }(\sup I_{n-1}(q)) \nonumber \\ &=&
\int_{I_{n-1}(q)} | x - c_{\inf I_{n-1}(q) , \sup I_{n-1}(q)} |^{r} d \mu (x) \nonumber \\
&<& 
\int_{I_{n-1}(p)} | x - c_{\inf I_{n-1}(p) , \sup I_{n-1}(p)} |^{r} d \mu (x) \nonumber \\
&=&
\int_{T(I_{n-1}(p))} 
| x - c_{-\sup I_{n-1}(p), - \inf I_{n-1}(p)}(\mu \circ T^{-1}) |^{r} d \mu \circ T^{-1} (x) \nonumber \\
&=&  W_{-\sup I_{n-1}(p),0, \mu \circ T^{-1} }( - \inf I_{n-1}(p)).
\label{strict_ieeq}
\end{eqnarray}
By (\ref{aibi}) we know that
$\inf I_{n-1}(q) = \sup I_{n-2}(q) < \sup I_{n-2}(p) = \inf I_{n-1}(p)$.
Thus, Lemma \ref{lemm_ex_01} implies
\begin{eqnarray}
W_{-\sup I_{n-1}(p),0, \mu \circ T^{-1} }( - \inf I_{n-1}(p)) & \leq &
W_{-\sup I_{n-1}(p),0, \mu \circ T^{-1} }( - \inf I_{n-1}(q)) \nonumber \\
&=& W_{\inf I_{n-1}(q),0, \mu }(\sup I_{n-1}(p)). 
\label{sup_inf_iequ}
\end{eqnarray}
Combing (\ref{strict_ieeq}) and (\ref{sup_inf_iequ}) we get
\[
W_{\inf I_{n-1}(q),0, \mu }(\sup I_{n-1}(q)) < W_{\inf I_{n-1}(q),0, \mu }(\sup I_{n-1}(p)).
\]
Now Lemma \ref{lemm_ex_01} yields 
\[
\sup I_{n-1}(q) < \sup I_{n-1}(p) \text{ and } \mu ( [ \sup I_{n-1}(q), \sup I_{n-1}(p) ] ) > 0.
\]
\smallskip \\
2. Rest of the proof. \smallskip \\
Applying (\ref{aibi}) and definition (\ref{def_wab}) we compute
\begin{equation}
\label{equ_01}
\int_{I_{n}(p)} | x - c_{\inf I_{n}(p) , \infty } |^{r} d \mu (x) =
W_{- \infty ,0, \mu \circ T^{-1} }(- \sup I_{n-1}(p)).
\end{equation}
Now, step 1, Lemma \ref{lemm_ex_01} and definition (\ref{def_wab}) implies
\begin{eqnarray}
W_{- \infty ,0, \mu \circ T^{-1} }(- \sup I_{n-1}(p)) 
& < &
W_{- \infty ,0, \mu \circ T^{-1} }(- \sup I_{n-1}(q)) \nonumber \\
&=&
\int_{I_{n}(q)} | x - c_{\inf I_{n}(q) , \infty } |^{r} d \mu (x).
\label{inf_inp}
\end{eqnarray}
From (\ref{contrad}) and (G4) we deduce
\begin{equation}
\label{inp_cinf}
\int_{I_{n}(q)} | x - c_{\inf I_{n}(q) , \infty } |^{r} d \mu (x) 
< \int_{I_{n}(p)} | x - c_{\inf I_{n}(p) , \infty } |^{r} d \mu (x).
\end{equation}
Combining (\ref{inp_cinf}), (\ref{inf_inp}) and (\ref{equ_01}) we get the contradiction
\[ 
\int_{I_{n}(p)} | x - c_{\inf I_{n}(p) , \infty } |^{r} d \mu (x) < 
\int_{I_{n}(p)} | x - c_{\inf I_{n}(p) , \infty } |^{r} d \mu (x),
\]
which proves the assertion. 
\end{proof}

Now we can state and prove the main result of this section.

\begin{proposition}
\emph{\textbf{(Existence and Uniqueness of Gersho quantizers)}}
\label{ref_prop_nempt}
Let $r>1$. For any $n \in \mathbb{N}$ the set $\mathcal{G}(n,\mu, r)$ of
Gersho-quantizers is not empty. Moreover,
if the support of $\mu$ is an interval, then 
exactly one Gersho quantizer exists, i.e.
$\card ( \mathcal{G}(n,\mu, r) ) =1$.
\end{proposition}
\begin{proof}
We proceed by induction.
Assume first that $n=1$. 
In this case the assertion follows directly from Proposition \ref{unique_1_quant}.
Now assume that for every $z \in V = \{ w \in \mathbb{R} : \mu ( I_{w, \infty} ) \in (0,1) \}$
the set $\mathcal{G}(n-1,\mu ( \cdot | I_{z, \infty} ), r)$ is not empty and that
$\card ( \mathcal{G}(n-1,\mu ( \cdot | I_{z, \infty} ), r) ) =1$ if the support 
of $\mu ( \cdot | I_{z, \infty} )$ is an interval. 
Let us consider the mapping 
\begin{equation}
\label{phi_def}
\Phi_{n-1}(z) 
= \frac{ \mu ( I_{z, \infty} ) D_{n-1,r}^{G}(\mu ( \cdot | I_{z, \infty} ) )}{n-1}.
\end{equation}
The idea of the proof is quite straightforward.
First, one constructs a quantizer which consists of $n$ cells, where the leftmost cell 
ends at $z$ and all other cells are designed such that every moment of one of these
$n-1$ cells equals $\Phi_{n-1}(z)$. 
Using the induction hypothesis it suffices to show that only one $z_{0}$ exists, such that
the moment of the leftmost cell and $\Phi_{n-1}(z_{0})$ coincide. 
Indeed, this can be achieved by showing that $\Phi_{n-1}(\cdot)$ is decreasing and
continuous.
We divide the remaining part of the proof into several steps. \smallskip \\
1. We construct quantizers $q_{n,z}$ and $\tilde{q}_{n,z}$
and show that $\tilde{q}_{n,z} \in \mathcal{G}(n-1,\mu ( \cdot | I_{z, \infty} ), r)$. 
We define for every $z \in V$ recursively 
\begin{eqnarray*}
&& w_{0}(z) = z , \qquad w_{n-1}(z) = \infty 
\end{eqnarray*}
and
\begin{eqnarray*}
w_{i}(z) = \inf \left\{ w > w_{i-1}(z) : 
\int_{[w_{i-1}(z),w]} | x - c_{w_{i-1}(z),w} |^{r} d \mu (x) = 
\Phi_{n-1}(z)   \right\} .
\end{eqnarray*}
for every $i \in \{ 1, \dots , n-2 \}$. 
Let $p_{z} \in \mathcal{G}(n-1,\mu ( \cdot | I_{z, \infty} ), r)$ and
denote by $I_{1}(p_{z}),\dots , I_{n-1}(p_{z})$ the codecells of $p_{z}$ 
in increasing order.
Note that 
\[
D_{n-1,r}^{G}(\mu ( \cdot | I_{z, \infty} ) ) = D( \mu ( \cdot | I_{z, \infty} ) , p_{z} , r )
\]
according to Proposition \ref{lemm_equal}.
Thus we deduce by (G4) and (\ref{phi_def}) that 
\[
\Phi_{n-1}(z) = \int_{[z,\sup I_{1}(p_{z} )]} | x - c_{z,\sup I_{1}(p_{z} )} |^{r} d \mu (x).
\]
Hence, the definition of $w_{1}(z)$ implies 
$w_{1}(z) \leq \sup I_{1}(p_{z})$. 
Moreover, 
\[
\mu ( [ w_{1}(z), \sup I_{1}(p_{z} ) ] ) = 0 ,
\]
because otherwise we would get together with Lemma \ref{lemm_ex_01} that
\begin{eqnarray*}
\frac{\mu ( I_{z, \infty} ) D( \mu ( \cdot | I_{z, \infty} ) , p_{z} , r )}{n-1}
&=& \Phi_{n-1}(z) \\
&=& \int_{[z , w_{1}(z)]} | x - c_{z, w_{1}(z)} |^{r} d \mu (x) \\
&<& \int_{[z, \sup I_{1}(p_{z})]} | x - p_{z}(x) |^{r} d \mu (x) ,
\end{eqnarray*}
which would contradict $p_{z} \in \mathcal{G}(n-1,\mu ( \cdot | I_{z, \infty} ), r)$.
Hence we obtain inductively that 
\begin{equation}
\nonumber
\mu ( [ w_{i}(z), \sup I_{i}(p_{z})]  ) = 0 \text{ for every } i \in \{ 1,..,n-2 \},
\end{equation}
which implies together with (G3), (G4) and (\ref{aibi}) that
\begin{eqnarray}
&& \int_{[w_{n-2}(z),w_{n-1}(z) )} |x-c_{w_{n-2}(z),w_{n-1}(z)}|^{r} d \mu (x)  \nonumber \\
&=&
\int_{I_{n-1}(p_{z})} |x-p_{z}(x)|^{r} d \mu (x) = \Phi_{n-1}(z) > 0.
\label{wnm2sdf}
\end{eqnarray}
From (\ref{wnm2sdf}) we know that $\mu ( [ w_{n-2}(z), w_{n-1}(z) ] ) > 0$. Moreover, the
definition of $w_{i}$ and Lemma \ref{lemm_ex_01} imply for $i \in \{ 1,\dots,n-2 \}$ that
\begin{equation}
\label{wiz_wiz}
\int_{[w_{i-1}(z), w_{i}(z)]} | x - c_{w_{i-1}(z), w_{i}(z)} |^{r} d \mu (x) = \Phi_{n-1}(z) > 0, 
\end{equation}
which yields $\mu ( [ w_{i-1}(z), w_{i}(z) ] ) > 0$ 
for every $i \in \{ 1,\dots,n-2 \}$. Thus we can define for $x \in \mathbb{R}$ the quantizers 
\[
q_{n,z}(x) = c_{- \infty, z} 1_{(- \infty, z )}(x)+
\sum_{i=0}^{n-2} c_{w_{i}(z),w_{i+1}(z)} 1_{[w_{i}(z), w_{i+1}(z))}(x) 
\]
and
\[
\tilde{q}_{n,z}(x) = c_{- \infty, w_{1}(z)} 1_{(- \infty, w_{1}(z) )}(x)+
\sum_{i=1}^{n-2} c_{w_{i}(z),w_{i+1}(z)} 1_{[w_{i}(z), w_{i+1}(z))}(x) .
\]
The definition of $q_{n,z}$ and $\tilde{q}_{n,z}$ imply together with
(\ref{wnm2sdf}) and (\ref{wiz_wiz}) that
\[
\tilde{q}_{n,z} \in \mathcal{G}(n-1,\mu ( \cdot | I_{z, \infty} ), r) .
\]
We will show that $\Phi_{n-1}( \cdot )$ is decreasing and continuous. \smallskip \\ 
2. We will show that $\Phi_{n-1}(\cdot)$ is decreasing. \smallskip \\
Let $\delta > 0$ such that $z + \delta < w_{1}(z)$
and $\mu ( [ z + \delta , w_{1}(z) ] ) > 0$. If 
$\mu ( [ z, z + \delta  ] ) = 0$, then definition (\ref{phi_def}) implies
$\Phi_{n-1}( z ) = \Phi_{n-1}( z + \delta )$. Hence we can assume w.l.o.g. that
$\mu ( [ z, z + \delta  ] ) > 0$. Thus we get from Lemma \ref{lemm_ex_01}
\begin{eqnarray}
\Phi_{n-1}(z) 
&=& \int_{ [z , w_{1}(z)] } | x - c_{z  , w_{1}(z)} |^{r} d \mu (x) \nonumber \\
&>& \int_{ [z + \delta , w_{1}(z)] } | x - c_{z  , w_{1}(z)} |^{r} d \mu (x) \nonumber \\
&\geq & \int_{ [z + \delta , w_{1}(z)] } | x - c_{z + \delta , w_{1}(z)} |^{r} d \mu (x) > 0.
\label{mon_Delt_v}
\end{eqnarray}
Now let us assume that 
\begin{equation}
\label{phinmv}
\Phi_{n-1}( z ) < \Phi_{n-1}( z + \delta ) .
\end{equation}
In this case
we obtain
\begin{eqnarray}
&& \int_{[z + \delta , w_{1}(z + \delta ) ]} | x - c_{z + \delta , w_{1}(z  )} |^{r} d \mu (x)  \nonumber \\
& \geq &
\int_{[z + \delta , w_{1}(z + \delta ) ]} | x - c_{z + \delta , w_{1}(z + \delta )} |^{r} d \mu (x) \nonumber  \\
&=& \Phi_{n-1}( z + \delta ) >  \Phi_{n-1}( z ).
\label{nm1phi}
\end{eqnarray}
Combining (\ref{mon_Delt_v}) and (\ref{nm1phi}) we get
\[
\int_{[z + \delta , w_{1}(z + \delta ) ]} | x - c_{z + \delta , w_{1}(z  )} |^{r} d \mu (x)
>
\int_{ [z + \delta , w_{1}(z)] } | x - c_{z + \delta , w_{1}(z)} |^{r} d \mu (x),
\]
which implies that $w_{1}(z) < w_{1}(z + \delta )$ and
$\mu ( [ w_{1}(z), w_{1}(z + \delta ) ] ) > 0$.
Inductively one obtains that $\mu ( [ w_{n-2}(z), w_{n-2}(z + \delta ) ] ) > 0$, yielding that
\begin{eqnarray*}
\Phi_{n-1}( z + \delta ) 
&=& \int_{[w_{n-2}(z + \delta ) , \infty )} | x - c_{w_{n-2}(z + \delta ) , \infty } |^{r} d \mu (x)   \\ 
&<&
\int_{[w_{n-2}(z ) , \infty )} | x - c_{w_{n-2}(z ) , \infty } |^{r} d \mu (x)  = \Phi_{n-1}( z ) ,
\end{eqnarray*}
which contradicts assumption (\ref{phinmv}). Thus we have proved that $\Phi_{n-1}( \cdot )$ is decreasing.
\smallskip \\
3. We will show that $\Phi_{n-1}(\cdot)$ is continuous. \smallskip \\
Note that $w_{1}(\cdot )$ is increasing. Indeed, 
Lemma \ref{lemm_ex_01} and Step 2 yields
\begin{eqnarray*}
\int_{[z , w_{1}(z + \delta )]} | x - c_{z, w_{1}(z + \delta )} |^{r} d \mu (x) 
&=& \Phi_{n-1}(z + \delta ) \leq \Phi_{n-1}(z) \\
&=& \int_{[z , w_{1}(z )]} | x - c_{z, w_{1}(z )} |^{r} d \mu (x) \\
& \leq &
\int_{[z , w_{1}(z )]} | x - c_{z, w_{1}(z + \delta )} |^{r} d \mu (x),
\end{eqnarray*}
implying $w_{1}(z + \delta ) \leq w_{1}(z )$.
Hence, $w_{1}(\cdot )$ is increasing and
consequently we get
\begin{eqnarray*}
\Phi_{n-1}( z + \delta  ) &=&  \int_{[z + \delta , w_{1}(z + \delta ) ]} 
| x - c_{z + \delta , w_{1}(z + \delta )} |^{r} d \mu (x) \\
& \geq &
\int_{[z + \delta , w_{1}(z  ) ]} | x - c_{z + \delta , w_{1}(z  )} |^{r} d \mu (x),
\end{eqnarray*}
which implies
\begin{eqnarray*}
0 &\leq & \Phi_{n-1}( z ) - \Phi_{n-1}( z + \delta  )  \\
& \leq &  
\int_{[z  , w_{1}(z  ) ]} | x - c_{z  , w_{1}(z  )} |^{r} d \mu (x) -
\int_{[z + \delta , w_{1}(z  ) ]} | x - c_{z + \delta , w_{1}(z  )} |^{r} d \mu (x) \\
& \leq &  
\int_{[z  , w_{1}(z  ) ]} | x - c_{z + \delta , w_{1}(z  )} |^{r} d \mu (x) -
\int_{[z + \delta , w_{1}(z  ) ]} | x - c_{z + \delta , w_{1}(z  )} |^{r} d \mu (x) \\
& \leq & \mu ( [z  , z + \delta ] ) | z - w_{1}(z )|^{r}.
\end{eqnarray*}
Because the right hand side tends to $0$ as $\delta \to 0$ we obtain that
$\Phi_{n-1}( \cdot )$ is continuous. \smallskip \\
4. Rest of the proof. \smallskip \\
Now we consider the mapping 
\[
\Psi_{1}(z) = 
W_{- \infty , 0, \mu }(z) = \int_{(- \infty, z)} | x - c_{-\infty, z}(\mu ) |^{r} d \mu (x) .
\]
From Lemma \ref{lemm_ex_01} we know that $\Psi_{1}( \cdot )$ is increasing and continuous.
Moreover, $\lim_{z \to - \infty}\Psi_{1}(z) = 0$ and $\lim_{z \to  \infty}\Psi_{1}(z) = D_{1,r}(\mu ) > 0$.
On the other hand we know from step 2 and 3 that $\Phi_{n-1}( \cdot )$ is decreasing and continuous.
Moreover, $\lim_{z \to - \infty} \Phi_{n-1}( z ) > 0$ and 
$\lim_{z \to  \infty} \Phi_{n-1}( z ) = 0$.
Hence a $z_{0} \in V$ exists, such that $\Phi_{n-1}( z_{0} ) = \Psi_{1}(z_{0})$.
It is easy to check that $q_{n,z_{0}} \in \mathcal{G}(n,\mu , r)$.
Finally, assume that the support of $\mu$ is an interval. 
By Lemma \ref{lemm_ex_01} the mapping $\Psi_{1}(\cdot)$ is strictly increasing in this case.
Hence, exactly one $z_{0} \in \mathbb{R}$ exists, such that $\Phi_{n-1}( z_{0} ) = \Psi_{1}(z_{0})$.
Moreover, we know by assumption that 
\[
\card ( \mathcal{G}(n-1,\mu ( \cdot | I_{z_{0}, \infty} ), r) ) =1 .
\]
From $\tilde{q}_{n,z_{0}} \in \mathcal{G}(n-1,\mu ( \cdot | I_{z, \infty} ), r)$ and by
the definition of $q_{n,z_{0}}$ we deduce, that $q_{n,z_{0}}$ is the only optimal $n-$th Gersho 
quantizer for $\mu$, i.e. $\card ( \mathcal{G}(n,\mu , r) ) =1$, which finishes the proof.
\end{proof}

\begin{remark}  
If $\mu$ is the uniform distribution on 
\[
[-1,-1/2] \cup [1/2, 1],
\]
then $\mathcal{G}(2,\mu, r)$ is not countable, because 
\[
q_{\varepsilon}(\cdot ) = 
-\frac{3}{4} \cdot 1_{(-\infty,-1/2 + \varepsilon )}(\cdot ) + 
\frac{3}{4} \cdot 1_{[-1/2 + \varepsilon, \infty )}(\cdot )
\]
is a Gersho quantizer for every $\varepsilon \in (0,1)$.
Let us identify two Gersho quantizers as equal if they have the same
codebook and for every codepoint the codecells are $\mu$-a.s. identical.
Insofar, we recognize that the quantizers $q_{\varepsilon}(\cdot )$ are all equal. 
Now it remains open, if a non-atomic distribution $\mu$ 
exists, where for some fixed $n$ and $r$ 
\begin{itemize}
\item[(a)]
at least two 
Gersho quantizers exist which are not equal but have the same codebook or 
\item[(b)]
at least two 
Gersho quantizers exist which have not the same codebook. 
\end{itemize}
Moreover,
it remains open to characterize those non-atomic scalar distributions where
for every $n$ and $r$ the corresponding Gersho quantizers are all equal. 
\end{remark}

\begin{remark}
\label{rem_mon_gerr}
Looking again at the proof of Proposition \ref{ref_prop_nempt}
we observe that $F(z) \to D_{n-1,r}^{G}(\mu )$ as $z \to - \infty$ and 
$F(z_{0}) = D_{n,r}^{G}(\mu )$, where the mapping $F$ is defined as
\[
(- \infty , z_{0} ] \ni z \to F(z)=\Psi_{1}(z)+(n-1)\Phi_{n-1}(z) .
\]
Unfortunately, the author was not able to prove or disprove the 
monotonicity of $F$. If $F$ would be decreasing, then
the overall distortion would be decreasing, i.e.
\begin{equation}
\label{strict_mon_ov_dist}
D_{n-1,r}^{G}(\mu ) \geq D_{n,r}^{G}(\mu )
\end{equation}
would be true.
Clearly, $D_{n,r}^{G}(\mu ) \geq D_{2n,r}^{G}(\mu )$, but it remains 
open if (\ref{strict_mon_ov_dist}) is true. 
Nevertheless, we will prove the weaker result that the distortion per
codecell is decreasing (cf. Lemma \ref{mom_mon} (b)).
\end{remark}

\begin{remark}
\label{rem_bisec}
Assume that the support of $\mu$ is a (possibly unbounded) interval
and denote by $I$ the interior of the support, which is an open interval.
Now we consider the mapping
\begin{eqnarray*}
I \ni x \to V(x) &=& \int_{(- \infty, x)} |z - c_{-\infty, x}|^{r} d \mu (z) 
- \int_{(x, \infty )} |z - c_{x, \infty}|^{r} d \mu (z) \\
&=& W_{-\infty, 0, \mu}(x) - W_{-\infty, 0, \mu \circ T^{-1}}(-x).
\end{eqnarray*}
From Lemma \ref{lemm_ex_01} we know that $V(\cdot)$ is continuous and strictly increasing.  
Let $x_{0} \in I$ be the unique point with $V(x_{0})=0$.
Clearly, $x_{0}$ is equal to the right resp. left boundary point of the codecells of
the unique $2$-level Gersho quantizer for $\mu$.
As sketched in the introduction, it is of practical relevance to determine 
$x_{0}$ by a bisection algorithm. If $r=2$, then it is well-known (cf. \cite[Example 2.3(b)]{GrLu00})
that $c_{-\infty, x}(\mu )$ equals the expected value of $\mu (  \cdot | (-\infty, x])$ for every $x \in I$.
Additionally, if $\mu$ has a continuous Lebesgue density, then
$V(\cdot)$ becomes differentiable on $I$ and one can use even a Newton type method for
the numerical determination of $x_{0}$.
\end{remark}

\section{High-rate quantization of scalar distributions with Gersho-quantizers}

We denote by $\lambda$ the Lebesgue measure on $\mathbb{R}$.
As already mentioned in the introduction, the asymptotic behavior of optimal quantizers 
- even in higher dimensions - is well-known. For scalar distributions the following holds.
Let $\mu = \mu_{a} + \mu_{s}$ be the Lebesgue decomposition of $\mu$ with respect to
$\lambda$, where $\mu_{a} = h_{a} \lambda$ with Lebesgue density $h_{a}$. 

\begin{theorem}[\cite{Za63, BuWi82, GrLu00}]
\label{stdres}
Assume that $\mu_{a}$ does not vanish and that 
\begin{equation}
\label{xrdetz}
\int |x|^{r + \delta} d \mu (x) < \infty
\end{equation}
for some $\delta > 0$. Then
\begin{equation}
\label{asympform}
n^{r} D_{n,r }(\mu ) \to 2^{-r}(1+r)^{-1} 
\left( \int h_{a}^{1/(1+r)} d \lambda \right)^{1+r} \qquad \text{if } n \to \infty .
\end{equation}
\end{theorem}  

Because we assume throughout this section that $\mu$ is absolutely continuous
with respect to $\lambda$, we make the 
following definition with $h$ as the Lebesgue density of $\mu$.
Let $Q(r)=2^{-r}(1+r)^{-1}$. If  
$ \int h^{1/(1+r)} d \lambda  < \infty$, then we define the constant
\[
C_{0} = Q(r) \left( \int h^{1/(1+r)} d \lambda \right)^{1+r}.
\]

\begin{remark}
\label{implfin}
Note that $C_{0} < \infty$ follows from (\ref{xrdetz}) by applying 
a H\"{o}lder argument, see \cite[Remark 6.3(a)]{GrLu00}.
\end{remark}

In view of (\ref{asympform}) one is interested in quantizer sequences
where the quantization error converges to $C_{0}$ as the 
quantization level tends to infinity.
To this end the following definition makes sense.

\begin{definition}
\label{defc0kj}
Assume that $C_{0} < \infty $.
We call a sequence $(q_{n})_{n \in \mathbb{N}}$ of quantizers asymptotically optimal if
$\card ( q_{n}(\mathbb{R}) ) \leq n$ and
\[
n^{r} D( \mu , q_{n} , r ) \to C_{0}
\]
for $n \to \infty $.
\end{definition}

\begin{remark}
Another definition of asymptotic optimality would be 
\[
\frac{D(\mu , q_{n} , r)}{D_{n,r}(\mu )} \to 1
\]
as $n \to \infty$, which is equivalent to Definition \ref{defc0kj} if 
(\ref{xrdetz}) is satisfied.
Slightly different from other definitions (cf. \cite[p.93]{GrLu00}), we require
in Definition (\ref{defc0kj}) only the finiteness of $C_{0}$, which is 
by \cite[Remark 6.3(b)]{GrLu00} weaker than condition (\ref{xrdetz}).
Clearly, if $(q_{n})_{n \in \mathbb{N}}$ is asymptotically optimal, then
$D( \mu , q_{n} , r ) \to 0$ as $n \to \infty$. 
\end{remark}

Now it is natural to ask if Gersho quantizers are asymptotically optimal.
As a main result of this paper (cf. Theorem \ref{main_theo}) we will prove in this section, that the answer is
positive at least for distributions with a so-called weakly unimodal density.

\begin{definition}
\label{def_weakly}
We call a probability density function $h$ weakly unimodal,
if $h$ is continuous on its support and there exists an $l_{0}> 0$ such that
$\{ x : h(x) \geq l \}$ is a compact interval for every $l \in (0, l_{0})$. 
\end{definition}

\begin{remark}
Many scalar probability densities are weakly unimodal,
e.g. the ones of Gaussian, Laplace or exponential distributions. 
Clearly, every distribution $\mu$ with a weakly unimodal density has
an (possibly unbounded) interval as support. Hence,
Proposition \ref{ref_prop_nempt} implies $\card ( \mathcal{G}(n , \mu , r) ) = 1$
if $r>1$. 
\end{remark}

In the proof of Theorem \ref{stdres} the assertion is proved first for 
distributions with piecewise constant density. Then the general case is 
proved by an approximation argument.
Unfortunately, we cannot use this technique for Gersho quantizers, because 
any amendments of the density will destroy property (G4) in general. 
Therefore, we use a different approach in this paper. 
We split the quantization error into the contribution of the codecells which
are lying inside a compact interval and the contribution of the remaining ones.
The asymptotic behavior inside the interval can be determined by  
the uniform continuity of $h$ on compact intervals. The number and the
contribution to the error of the codecells which are outside the interval can
be controlled by (G4).  

Let $- \infty < u < v < \infty$ and $I=[u,v]$. Assume that $\mu (I) > 0$.
For every $n \in \mathbb{N}$ let
$n_{2}^{I}(n)$ be the number of codepoints whose codecell is located outside $(u,v)$, i.e.
\begin{equation}
\label{n2in}
n_{2}^{I}(n) = \card ( \{ c \in q_{n}(\mathbb{R}) : q_{n}^{-1}(c) \subset \mathbb{R} \backslash (u,v) \} ).
\end{equation}
It should be noted at this point that $n_{1}^{I}(n)$ will be defined and used later, cf. (\ref{jni_Def}).

The following result shows that $n_{2}^{I}( \cdot )$ is increasing.
By doing this, we first prove that the distortion per codecell is decreasing (cf. Lemma \ref{mom_mon} (b)). 
As already mentioned in Remark \ref{rem_mon_gerr} this result is weaker than 
relation (\ref{strict_mon_ov_dist}).

\begin{lemma}
\label{mom_mon}
Let $\mu$ be weakly unimodal, $r>1$ 
and $(q_{n})_{n \in \mathbb{N}}$ the sequence of $n$-level Gersho quantizers for $\mu$.
Then, 
\begin{itemize}
\item[(a)] 
$n_{2}^{I}(n+1) \geq n_{2}^{I}(n)$ and
\item[(b)] 
$\frac{D(\mu , q_{n}, r)}{n} \geq \frac{D(\mu , q_{n+1}, r)}{n+1}$
for every $n \in \mathbb{N}$.
\end{itemize}
\end{lemma}
\begin{proof}
First we will prove assertion (b). \smallskip \\
If $n=1$, then let $I_{1}(q_{2})$ and $I_{2}(q_{2})$  
in increasing order.
Applying (\ref{aibi}) and (G3) we calculate
\begin{eqnarray*}
&& D( \mu , q_{1} , r ) \\
&=& \int | x - q_{1}(x) |^{r} d \mu (x) \\
&=&
\int_{(- \infty , \sup I_{1}(q_{2}))} | x - q_{1}(x) |^{r} d \mu (x) +
\int_{[ \sup I_{1}(q_{2}) ,  \infty )} | x - q_{1}(x) |^{r} d \mu (x) \\
& \geq &
\int_{I_{1}(q_{2}) } | x - c_{- \infty , \sup I_{1}(q_{2})} |^{r} d \mu (x) +
\int_{I_{2}(q_{2}) } | x - c_{\inf I_{2}(q_{2}) ,  \infty } |^{r} d \mu (x) \\
&=& D( \mu , q_{2} , r ) > D( \mu , q_{2} , r ) / 2 .
\end{eqnarray*}
Now let $n \geq 2$.
In this case we proceed indirectly and assume the contrary, i.e. we assume that
\begin{equation}
\label{dmqno}
\frac{D(\mu , q_{n}, r)}{n} < \frac{D(\mu , q_{n+1}, r)}{n+1}.
\end{equation}
Now let $q_{n}(\mathbb{R}) = \{ c_{1},\dots,c_{n} \}$ with $- \infty < c_{1} < \dots < c_{n} < \infty$.
Moreover, let
$q_{n+1}(\mathbb{R}) = \{ d_{1},\dots,d_{n+1} \}$ with $- \infty < d_{1} < \dots < d_{n+1} < \infty$.
Let $y_{1}=\sup q_{n}^{-1}(c_{1})$ and $z_{1}=\sup q_{n+1}^{-1}(d_{1})$.
From (\ref{dmqno}) we obtain together with (G4) that
\[
\int_{(-\infty , y_{1})} |x - c_{1}|^{r} d \mu (x) < \int_{(-\infty , z_{1})} |x - d_{1}|^{r} d \mu (x) .
\]
Consequently, Lemma \ref{lemm_ex_01} implies together with (G3)  
that $y_{1} < z_{1}$.
If $n > 2$, then we obtain inductively that
\begin{equation}
\label{qnobi}
\sup q_{n}^{-1}(c_{i}) < \sup q_{n+1}^{-1}(d_{i}) \leq  \sup q_{n+1}^{-1}(d_{i+1}) 
\end{equation}
for every $i \in \{1,..,n-1 \}$.
On the other hand, applying again (\ref{dmqno}) and 
(G4) we get 
\[
\int_{( \inf q_{n+1}^{-1}(d_{n+1}) , \infty )} |x - d_{n+1}|^{r} d \mu (x) 
> \int_{( \inf q_{n}^{-1}(c_{n}) , \infty )} |x - c_{n}|^{r} d \mu (x) .
\]
Thus, (\ref{aibi}) and Lemma \ref{lemm_ex_01} yields 
\[
\sup q_{n+1}^{-1}(d_{n}) = 
\inf q_{n+1}^{-1}(d_{n+1}) <  \inf q_{n}^{-1}(c_{n}) = \sup q_{n}^{-1}(c_{n-1}),
\]
which contradicts (\ref{qnobi}).
Thus, assertion (b) is proved. \smallskip \\
Now we will prove assertion (a). \smallskip \\
For any $n \in \mathbb{N}$ the number $n_{2}^{I}(n)$ can be divided into two parts.
The number of codecells on the left of interval $I$ and right to $I$.
Thus we define
\[
n_{2}^{I,l}(n) = \card \{ c \in q_{n}(\mathbb{R}) : \sup q_{n}^{-1}(c) \leq \min I  \} 
\]
and
\[
n_{2}^{I,r}(n) = \card \{ c \in q_{n}(\mathbb{R}) : \inf q_{n}^{-1}(c) \geq \max I  \} .
\]
If $n_{2}^{I,l}(n)=0$, then we immediately obtain $n_{2}^{I,l}(n) \leq n_{2}^{I,l}(n+1)$.
Now let us assume that $n_{2}^{I,l}(n) > 0$.
Let $a = \min q_{n}(\mathbb{R})$  and $b = \min q_{n+1}(\mathbb{R})$.
From (b) we obtain
\[
\int_{q_{n}^{-1}(a)} |x - a|^{r} d \mu (x) \geq  \int_{q_{n+1}^{-1}(b)} |x - b|^{r} d \mu (x) .
\]
Applying (\ref{aibi}) and (G3) we obtain from Lemma \ref{lemm_ex_01}  
that $\sup q_{n}^{-1}(a) \geq \sup q_{n+1}^{-1}(b)$. 
Now let us define $a^{\prime}$ resp. $b^{\prime}$ as the 
largest codepoint of $q_{n}(\mathbb{R})$ resp. $q_{n+1}(\mathbb{R})$ which
is located left to $I$. More exactly,
\[
a^{\prime} = \max \{ c \in q_{n}(\mathbb{R}) : \sup q_{n}^{-1}(c) \leq \min I  \}
\]
and
\[
b^{\prime} = \max \{ c \in q_{n+1}(\mathbb{R}) : \sup q_{n+1}^{-1}(c) \leq \min I  \} .
\]
Thus, using Lemma \ref{lemm_ex_01} again, 
we obtain 
by induction 
that $\sup q_{n}^{-1}(a^{\prime}) \geq \sup q_{n+1}^{-1}(b^{\prime})$, which implies
$n_{2}^{I,l}(n) \leq n_{2}^{I,l}(n+1)$.
By the same argumentation we deduce that $n_{2}^{I,r}(n) \leq n_{2}^{I,r}(n+1)$.
Assertion (b) follows now by the observation
$n_{2}^{I,l}(n) + n_{2}^{I,r}(n) = n_{2}^{I}(n)$ and 
the proof is finished.
\end{proof}

For any $n-$level Gersho quantizer $q_{n}$ we define

\begin{equation}
\label{jni_Def}
J_{n, I} = \{ a \in q_{n}(\mathbb{R}) :  q_{n}^{-1}(a) \subset I  \} 
\qquad \text{and} \quad n_{1}^{I}(n) = \card J_{n, I}.
\end{equation}
Moreover, we denote by $S(\mu )$ the interior of $\supp(\mu )$ and define
$T_{n,I}$ as the closure of the union of all the $n-$level codecells in $I$, i.e. 
\[
T_{n,I}= [ \inf q_{n}^{-1}(\min J_{n,I}) ,  \sup q_{n}^{-1}(\max J_{n,I})  ]
\]
if $J_{n, I}$ is nonempty.
As already depicted above, we first prove the asymptotic optimality for
the part of the quantization error where codecells are located inside a 
compact interval (cf. Lemma \ref{lemm_i1}). As a motivation for our approach let us consider the following
heuristic argumentation. For large $n$ and $a \in J_{n,I}$ we observe that  
\begin{eqnarray}
&& Q(r) \left(  \int_{q_{n}^{-1}(a)} h^{1/(1+r)} d \lambda \right)^{1+r} 
\approx 
\diam ( q_{n}^{-1}(a) )^{1+r} h(a) Q(r) \nonumber \\
&=& 
\int_{q_{n}^{-1}(a)} h(a) |x-a|^{r} d \lambda (x) 
\approx 
\int_{q_{n}^{-1}(a)} |x-a|^{r} d \mu (x) = n^{-1} D_{n,r}^{G}(\mu ),
\label{heur_01}
\end{eqnarray}
where the last equation is due to (G4).
The right hand side is independent of $a \in J_{n,I}$ and, therefore, also approximately the 
left hand side. But this implies that 
\begin{equation}
\label{heur_02}
Q(r) \left(  \int_{q_{n}^{-1}(a)} h^{1/(1+r)} d \lambda \right)^{1+r} 
\approx
Q(r) \left(  (n_{1}^{I}(n))^{-1} \int_{I} h^{1/(1+r)} d \lambda \right)^{1+r}.
\end{equation}
Now (\ref{heur_01}) and (\ref{heur_02}) indicate that
\begin{equation}
\label{heur_app_03}
n^{r} D_{n,r}^{G}(\mu ) \approx 
Q(r) n^{1+r} (n_{1}^{I}(n))^{-(1+r)} \left( \int_{I} h^{1/(1+r)} d \lambda \right)^{1+r}.
\end{equation}
To show exactly that asymptotic optimality holds, we first have to justify the approximation in (\ref{heur_01}) 
and (\ref{heur_02}). This is done by Lemma \ref{lemm_w_u_w_u}. 
Moreover, we must show that the right hand side of (\ref{heur_app_03}) converges against
a constant $C(I)$ and that $|C(I) - C_{0}|$ is arbitrarily small according to a suitable choice of 
$I \subset \supp(\mu)$. This is the content of Lemma \ref{lemm_i1} and the proof of Theorem \ref{main_theo}.      
\begin{lemma}
\label{lemm_w_u_w_u}
Let $r>1$.
Let $\mu$ be weakly unimodal and $(q_{n})_{n \in \mathbb{N}}$ the sequence of $n-$level Gersho quantizers for $\mu$.
Assume that $I \subset S ( \mu )$. Then an $m \in \mathbb{N}$ exists, such that $J_{n, I}$ 
is nonempty for every $n \geq m$.
Let $n \geq m$ and 
\begin{eqnarray*}
\varepsilon_{1}(n)  &=& \max \left\{ \left| \frac{\int_{q_{n}^{-1}(a)} h^{1/(1+r)} 
d \lambda}{\int_{T_{n,I}} h^{1/(1+r)} d \lambda} - \frac{1}{n_{1}^{I}(n)} \right| : 
a \in J_{n, I}  \right\}, \\
\varepsilon_{2}(n) &=& \max \{ 
| \int_{q_{n}^{-1}(a)} | x - a |^{r} d \mu (x) - \diam ( q_{n}^{-1}(a)  )^{1+r} h(a) Q(r) |
: a \in J_{n, I}  \}  , \\
\varepsilon_{3}(n) &=& \max \{ 
| \int_{q_{n}^{-1}(a)} h^{1/(1+r)} d \lambda  - \diam ( q_{n}^{-1}(a)  ) h(a)^{1/(1+r)} |
: a \in J_{n, I}  \} .
\end{eqnarray*}
Then,
\begin{eqnarray}
\diam ( T_{n,I} ) & \to & \diam (I), 
\label{equ_neps00} \\
n_{1}^{I}(n) \varepsilon_{1}(n) & \to & 0 , 
\label{equ_neps01} \\
(n_{1}^{I}(n))^{1+r} \varepsilon_{2}(n) & \to & 0 , 
\label{equ_neps02} \\
n_{1}^{I}(n) \varepsilon_{3}(n) & \to & 0,
\label{equ_neps03} 
\end{eqnarray}
as $n \to \infty$.
Moreover,
\begin{equation}
\label{neps3}
\limsup_{n \to \infty} \frac{\varepsilon_{3}(n)}{ \min \{ \diam ( q_{n}^{-1}(a)  ) : a \in J_{n, I} \} } = 0.
\end{equation}
\end{lemma}
\begin{proof}
Because $\mu$ is weakly unimodal and due to $I \subset S (\mu )$ we get
\begin{equation}
\label{m1iert}
0 < M_{1,I} := \min \{ h(x) : x \in I \} \text{ and } 
\end{equation}
\[
\infty > M_{2,I} := \max \{ h(x) : x \in I \} \geq M_{1,I} .
\]
We divide the remaining proof into six steps. \smallskip \\
1. We will show that $J_{n, I}$ is nonempty for every $n \geq m$. \smallskip \\
Note that  
\begin{eqnarray}
D(\mu , q_{n} , r) &=& \sum_{a \in q_{n}(\mathbb{R})} \int_{q_{n}^{-1}(a)} |x-a|^{r} d \mu (x)  \nonumber \\
& \leq & \sum_{a \in q_{n}(\mathbb{R})} \int_{q_{n}^{-1}(a)} |x-q_{1}(x)|^{r} d \mu (x) =
D(\mu , q_{1} , r) < \infty ,
\label{equ_1_n}
\end{eqnarray}
which implies 
\begin{equation}
\label{dmuqnertzxcv}
D(\mu , q_{n} , r) / n \to 0
\end{equation}
as $n \to \infty$. If $J_{n, I} = \emptyset$, then 
a $c \in q_{n}(\mathbb{R})$ exists, such that 
\[
\diam ( q_{n}^{-1}(c) \cap I ) \geq  \diam (I)/2 \text{ and }
q_{n}^{-1}(c) \backslash  I \neq \emptyset .
\]
Recall $I=[u,v]$.
Hence, (G4) yields 
\begin{equation}
\label{e1e234}
D( \mu , q_{n} , r) / n \geq \min \{ e_{1} , e_{2} \} > 0
\end{equation}
with
\[
e_{1} = \int_{[u , (u+v)/2 ]} |x - ( \frac{3}{4} u + \frac{1}{4} v ) |^{r} d \mu (x)
\]
and
\[
e_{2} = \int_{[(u+v)/2 , v ]} |x - ( \frac{1}{4} u + \frac{3}{4} v ) |^{r} d \mu (x) .
\]
Note, that the right hand side of (\ref{e1e234}) is independent of $n$. Thus, we deduce from 
(\ref{dmuqnertzxcv}) and (\ref{e1e234}) that  
an $m \in \mathbb{N}$ exists, such that $J_{n, I}$ is nonempty for all $n \geq m$.
Let us assume w.l.o.g. for the rest of this proof that $m=1$. 
\smallskip \\
2. We will prove (\ref{equ_neps00}). \smallskip \\
We proceed indirectly and assume that a constant $M \in [0 , \diam (I))$ and a
subsequence of $(\diam ( T_{n,I} ))_{n \in \mathbb{N}}$ exists, which
we also denote by $(\diam ( T_{n,I} ))_{n \in \mathbb{N}}$, such that 
\begin{equation}
\label{tniert}
\diam ( T_{n,I} ) \to M \quad \text{ as } n \to \infty .
\end{equation}
Let 
\begin{eqnarray*}
&& s_{n} = \max\{ a \in q_{n}(\mathbb{R}) : \inf q_{n}^{-1}(a) < \min I \} , \\
&& t_{n} = \min\{ a \in q_{n}(\mathbb{R}) : \sup q_{n}^{-1}(a) > \max I \}.
\end{eqnarray*}
Let 
\[
s = \lim_{n \to \infty } ( \sup q_{n}^{-1}(s_{n}) - \min I ) \text{ and } 
t = \lim_{n \to \infty } ( \max I - \inf q_{n}^{-1}(t_{n}) ) ). 
\]
Due to (\ref{tniert}) we know that $s$ and $t$ exist and that
$s+t = \diam I - M > 0$.
If we assume w.l.o.g. that $s>0$, then we obtain
\begin{eqnarray}
&& \liminf_{n \to \infty } \int_{q_{n}^{-1}(s_{n})} | x - s_{n} |^{r} d \mu (x) \nonumber \\
& \geq &
\liminf_{n \to \infty } \int_{I \cap q_{n}^{-1}(s_{n})} | x - s_{n} |^{r} M_{1,I} d \lambda (x) \nonumber \\
& \geq & 
\liminf_{n \to \infty } \int_{I \cap q_{n}^{-1}(s_{n})} 
| x - (\sup q_{n}^{-1}(s_{n}) + \min I)/2 |^{r} M_{1,I} d \lambda (x)  \nonumber \\
&=& 
\liminf_{n \to \infty } Q(r) M_{1,I}( \sup q_{n}^{-1}(s_{n}) - \min I )^{1+r} = Q(r) M_{1,I} s^{1+r} > 0.
\label{scfghjk}
\end{eqnarray}
On the other hand we get from (G4) 
that
\begin{eqnarray}
\limsup_{n \to \infty } \int_{q_{n}^{-1}(s_{n})} | x - s_{n} |^{r} d \mu (x) &=& 
\limsup_{n \to \infty } \frac{D(\mu , q_{n} , r)}{n} .
\label{equna_rn01}
\end{eqnarray}
Combining (\ref{equna_rn01}) and (\ref{equ_1_n}) we obtain
\[
\limsup_{n \to \infty } \int_{q_{n}^{-1}(s_{n})} | x - s_{n} |^{r} d \mu (x) = 0,
\]
which contradicts (\ref{scfghjk}).
\smallskip \\
3. We will show that $n_{0} \in \mathbb{N}$ and constants $A_{0}, B_{0} \in (0, \infty)$ exist, such that 
\[
\max \{ \diam ( q_{n}^{-1}(a)  ) : a \in J_{n, I} \} \leq A_{0} (n_{1}^{I}(n))^{-1}
\]
and
\[
\min \{ \diam ( q_{n}^{-1}(a)  ) : a \in J_{n, I} \} \geq B_{0} (n_{1}^{I}(n))^{-1}
\] 
for every  $n \geq n_{0}$. \smallskip \\
Let $a \in J_{n, I}$ and notice that
\begin{eqnarray}
&& M_{1,I} \leq M_{1,I}^{a} = 
\min \{ h(x) : x \in q_{n}^{-1}(a) \} , \nonumber \\
&& M_{2,I} \geq M_{2,I}^{a} = \max \{ h(x) : x \in q_{n}^{-1}(a) \}.
\label{def_m1i}
\end{eqnarray}
Applying property (G3) of $q_{n}$ we obtain that
\begin{eqnarray}
Q(r) M_{1,I}^{a} \diam ( q_{n}^{-1}(a)  )^{1+r} &=&
\inf \{ \int_{q_{n}^{-1}(a) } | x - b |^{r} M_{1,I}^{a} d \lambda (x) : b \in \mathbb{R} \} \nonumber \\
& \leq &  \int_{q_{n}^{-1}(a) } | x - a |^{r} d \mu (x) \nonumber \\
& \leq & 
\inf \{ \int_{q_{n}^{-1}(a)} | x - b |^{r} M_{2,I}^{a} d \lambda (x) : b \in \mathbb{R} \} \nonumber \\
&=& 
Q(r) M_{2,I}^{a} \diam ( q_{n}^{-1}(a)  )^{1+r}. 
\label{rel_msdf} 
\end{eqnarray}
Now (G4) and (\ref{def_m1i}) are implying that
\begin{eqnarray}
&& M_{1,I} \left( \max \{ \diam ( q_{n}^{-1}(b)  ) : b \in J_{n, I} \} \right) ^{1+r} \nonumber \\
& \leq &
Q(r)^{-1}  \max \{ \int_{q_{n}^{-1}(b) } | x - b |^{r} d \mu (x) : b \in J_{n, I} \} \nonumber \\
&=&
Q(r)^{-1}  \min \{ \int_{q_{n}^{-1}(b) } | x - b |^{r} d \mu (x) : b \in J_{n, I} \} \nonumber \\
& \leq &
M_{2,I} \left( \min \{ \diam ( q_{n}^{-1}(b)  ) : b \in J_{n, I} \} \right) ^{1+r}.
\label{low_up_m1m2}
\end{eqnarray}
According to (\ref{rel_msdf}) a $\xi_{a} \in q_{n}^{-1}(a) \subset I$ exists, such that
\begin{equation}
\label{repr_a_nerr}
\int_{q_{n}^{-1}(a) } | x - a |^{r} d \mu (x) = Q(r) h(\xi_{a}) \diam ( q_{n}^{-1}(a)   )^{1+r} .
\end{equation}
From (\ref{low_up_m1m2}) we deduce that
\begin{eqnarray}
&& n_{1}^{I}(n) \cdot \min \{ \diam ( q_{n}^{-1}(b)  ) : b \in J_{n, I} \} \leq \diam (T_{n, I}) \nonumber \\
& \leq &
n_{1}^{I}(n) \cdot \max \{ \diam ( q_{n}^{-1}(b)  ) : b \in J_{n, I} \} \nonumber \\
& \leq &
n_{1}^{I}(n) \cdot \left( \frac{ M_{2,I}}{ M_{1,I}} \right)^{1/(1+r)}  
\min \{ \diam ( q_{n}^{-1}(b)   ) : b \in J_{n, I} \} \nonumber \\
& \leq &
\left( \frac{ M_{2,I}}{ M_{1,I}} \right)^{1/(1+r)} \diam (T_{n, I}),
\label{hill_step}
\end{eqnarray}
where the last inequality follows from the first one. 
According to (\ref{equ_neps00}) and definition (\ref{jni_Def}) choose $n_{0} \in \mathbb{N}$ such that
\[
\diam (I) / 2 \leq \diam (T_{n, I}) \leq \diam (I)
\]
for every $n \geq n_{0}$.
Hence, the setting 
\[
A_{0} =  \left( \frac{ M_{2,I}}{ M_{1,I}} \right)^{1/(1+r)} \diam (I), \qquad 
B_{0} = \left( \frac{ M_{1,I}}{ M_{2,I}} \right)^{1/(1+r)} \diam (I)/2
\]
and relation (\ref{hill_step}) finishes the proof of this step. \smallskip \\
4. We will prove (\ref{equ_neps02}). \smallskip \\
Let us denote by $\kappa_{I}$ the modulus of continuity
of $h$ restricted to $I$, i.e.
\[
\kappa_{I}(w) = \sup \{ | h(x) - h(y) | : x,y \in I , |x-y| \leq w \}.
\]
Applying (\ref{repr_a_nerr}) and Step 3 we deduce for every $n \geq n_{0}$ that
\begin{eqnarray}
&& \max \{ 
| \int_{q_{n}^{-1}(a)} | x - a |^{r} d \mu (x) - \diam ( q_{n}^{-1}(a) )^{1+r} h(a) Q(r) |
: a \in J_{n, I}  \} \nonumber \\
& \leq & Q(r) \max \{ \kappa_{I} ( \diam ( q_{n}^{-1}(a) ) ) 
\diam ( q_{n}^{-1}(a)   )^{1+r} : a \in J_{n, I}  \}  \nonumber  \\
&\leq &  Q(r) \kappa_{I} ( A_{0} (n_{1}^{I}(n))^{-1} ) (A_{0} (n_{1}^{I}(n))^{-1})^{1+r}.
\label{kapizu}
\end{eqnarray}
Clearly, according to (\ref{equ_1_n}) we know that 
\begin{equation}
\label{ctozer}
D(\mu , q_{n} , r) / n \to 0  \qquad \text{as }  n \to \infty .
\end{equation}
This implies that 
\begin{equation}
\label{convinftz}
n_{1}^{I}(n) \to \infty \qquad \text{as } n \to \infty .
\end{equation}
Indeed, if we assume the contrary, then
we can find a sequence $(v_{n})_{n \in \mathbb{N}}$ with $v_{n} \in J_{n,I}$ such that
\[
\int_{q_{n}^{-1}(v_{n})} | x - v_{n} |^{r} 
d \mu \geq Q(r) M_{1,I} \liminf_{n \to \infty } \diam ( q_{n}^{-1}(v_{n})  )^{1+r} > 0, 
\]
which implies according to property (G4) of a Gersho quantizer that 
\[
\liminf_{n \to \infty } D(\mu , q_{n} , r) / n > 0,
\]
a contradiction to (\ref{ctozer}).
Thus, (\ref{convinftz}) implies together with (\ref{kapizu}) and the definition of $\varepsilon_{2}(n)$ that 
$(n_{1}^{I}(n))^{1+r} \varepsilon_{2}(n) \to 0$ as $n \to \infty$. \smallskip \\
5. We will prove (\ref{equ_neps03}) and (\ref{neps3}). \smallskip \\
For every $a,b \in J_{n, I}$ a $\tau_{a} \in q_{n}^{-1}(a)$ and $\tau_{b} \in q_{n}^{-1}(b)$ 
exist, such that
\begin{equation}
\label{taua}
h(\tau_{a})^{1/(1+r)} \diam ( q_{n}^{-1}(a)  ) = \int_{q_{n}^{-1}(a)} h^{1/(1+r)} d \lambda
\end{equation}
and
\begin{equation}
\label{taub}
h(\tau_{b})^{1/(1+r)} \diam ( q_{n}^{-1}(b)  ) = \int_{q_{n}^{-1}(b)} h^{1/(1+r)} d \lambda .
\end{equation}
Moreover,
\begin{eqnarray*}
&& | \int_{q_{n}^{-1}(a)} h^{1/(1+r)} d \lambda - \diam ( q_{n}^{-1}(a) ) h(a)^{1/(1+r)} | \\
&=& | h(\tau_{a})^{1/(1+r)}  -  h(a)^{1/(1+r)} |\diam ( q_{n}^{-1}(a)  ) .
\end{eqnarray*}
If we denote by $\omega_{I}$ the modulus of continuity of $h^{1/(1+r)}$ restricted
to $I$, then step 3 yields for every $n \geq n_{0}$ that 
\begin{eqnarray*}
\varepsilon_{3}(n) &=& \max \{ | \int_{q_{n}^{-1}(a)} h^{1/(1+r)} d \lambda 
- \diam ( q_{n}^{-1}(a)  ) h(a)^{1/(1+r)} | : a \in J_{n,I} \} \\
& \leq &  \omega_{I} ( A_{0} (n_{1}^{I}(n))^{-1} ) A_{0} (n_{1}^{I}(n))^{-1} .
\end{eqnarray*}
Because $h^{1/(1+r)}$ is uniformly continuous on $I$ we have 
$n_{1}^{I}(n) \varepsilon_{3}(n) \to 0$ as $n \to \infty$. 
Moreover, (\ref{neps3}) follows immediately from step 3.
\smallskip \\
6. We will prove (\ref{equ_neps01}). \smallskip \\
Let $n \geq n_{0}$ and $a,b \in J_{n, I}$ and let $\xi_{a} \in q_{n}^{-1}(a)$ and $\xi_{b} \in q_{n}^{-1}(b)$ 
according to (\ref{repr_a_nerr}). Let $\tau_{a} \in q_{n}^{-1}(a)$ and $\tau_{b} \in q_{n}^{-1}(b)$
according to (\ref{taua}) and (\ref{taub}).
We define 
\[
f(n) = \max \left\{ \left| \left( \frac{h(\tau_{c})}{h(\tau_{d})} \right)^{1/(1+r)} 
\left( \frac{h(\xi_{d})}{h(\xi_{c})} \right)^{1/(1+r)} - 1 \right| :  c,d \in J_{n, I} \right\} . 
\]
From (\ref{ctozer}) we obtain by the same argumentation as for (\ref{convinftz}) that 
\[
\sup \{ \diam ( q_{n}^{-1}(c) ) : c \in J_{n,I} \} \to 0 \qquad \text{as } n \to \infty .
\]
But this implies together with the continuity of $h$ that $f(n) \to 0$ as $n \to \infty$.
Now we calculate with (\ref{taua}) and (\ref{taub}) that 
\begin{eqnarray*}
&& | \int_{q_{n}^{-1}(a)} h^{1/(1+r)} d \lambda - \int_{q_{n}^{-1}(b)} h^{1/(1+r)} d \lambda | \\
&=&
| h(\tau_{a})^{1/(1+r)} \diam ( q_{n}^{-1}(a)   ) - h(\tau_{b})^{1/(1+r)} \diam ( q_{n}^{-1}(b)   ) | .
\end{eqnarray*}
Let us assume w.l.o.g. that 
\[
h(\tau_{a})^{1/(1+r)} \diam ( q_{n}^{-1}(a)   ) \geq h(\tau_{b})^{1/(1+r)} \diam ( q_{n}^{-1}(b)  ).
\]
Thus we obtain from property (G4), relation (\ref{repr_a_nerr}) and step 3 that
\begin{eqnarray}
&& | \int_{q_{n}^{-1}(a)} h^{1/(1+r)} d \lambda - \int_{q_{n}^{-1}(b)} h^{1/(1+r)} d \lambda | \nonumber \\
&=&
h(\tau_{a})^{1/(1+r)} \diam ( q_{n}^{-1}(a)  ) - 
h(\tau_{b})^{1/(1+r)} \diam ( q_{n}^{-1}(b)  ) \nonumber \\
&=&
h(\tau_{a})^{1/(1+r)} \diam ( q_{n}^{-1}(b)  ) \left( \frac{h(\xi_{b})}{h(\xi_{a})} \right)^{1/(1+r)} 
\nonumber \\
&& \qquad \qquad \qquad \qquad  - h(\tau_{b})^{1/(1+r)} \diam ( q_{n}^{-1}(b)   ) \nonumber \\
& \leq & \diam ( q_{n}^{-1}(b)   ) h( \tau_{b} )^{1/(1+r)} f(n) \nonumber \\
&\leq &
A_{0} \frac{1}{n_{1}^{I}(n)} M_{2,I}^{1/(1+r)} f(n) .
\label{def_naom1}
\end{eqnarray}
Now we define for every $c \in J_{n, I}$
\[
p_{c} = \frac{\int_{q_{n}^{-1}(c)} h^{1/(1+r)} d \lambda}{\int_{T_{n,I}} h^{1/(1+r)} d \lambda} .
\]
Applying (\ref{def_naom1}) we get
\begin{eqnarray} 
&& |\frac{1}{n_{1}^{I}(n)}  - p_{a} | 
= |\frac{1}{n_{1}^{I}(n)} \sum_{c \in J_{n,I}} ( p_{c} - p_{a} ) | \nonumber \\
&\leq & \frac{1}{n_{1}^{I}(n)}  \sum_{c \in J_{n,I}} 
\frac{| \int_{q_{n}^{-1}(a)} h^{1/(1+r)} d \lambda - \int_{q_{n}^{-1}(c)} h^{1/(1+r)} d \lambda |}
{\int_{T_{n,I}} h^{1/(1+r)} d \lambda } \nonumber \\
& \leq & \left( \int_{T_{n,I}} h^{1/(1+r)} d \lambda \right)^{-1} 
\frac{A_{0}}{n_{1}^{I}(n)} M_{2,I}^{1/(1+r)} f(n).
\label{123wert}
\end{eqnarray}
Now, (\ref{123wert}) and (\ref{equ_neps00}) imply (\ref{equ_neps01}) and the proof is finished.
\end{proof}

Recall $Q(r)=2^{-r}(1+r)^{-1}$ and $S(\mu )$ as the interior of $\supp ( \mu )$.

\begin{lemma}
\label{lemm_i1}
Let $\mu$ be weakly unimodal and $(q_{n})_{n \in \mathbb{N}}$ the sequence of $n-$level Gersho quantizers for $\mu$.
Assume that $I \subset S ( \mu )$.
Then,
\[
n_{1}^{I}(n)^{r} \sum_{a \in J_{n, I}} \int_{q_{n}^{-1}(a)} | x - a |^{r} d \mu (x) \to  Q(r) 
\left( \int_{I} h^{1/(1+r)} d \lambda \right)^{1+r} 
\]
as $n \to \infty$.
\end{lemma}
\begin{proof}
We deduce from (\ref{equ_neps02}) and the definition of $\varepsilon_{3}(n)$ that
\begin{eqnarray}
&& \limsup_{n \to \infty} n_{1}^{I}(n)^{r} 
\sum_{a \in J_{n, I}} \int_{q_{n}^{-1}(a)} | x - a |^{r} d \mu (x)  \nonumber \\
& \leq &
\limsup_{n \to \infty} n_{1}^{I}(n)^{r} \left( n_{1}^{I}(n) \varepsilon_{2}(n) + 
\sum_{a \in J_{n, I}} \diam( q_{n}^{-1}(a)  )^{1+r} h(a) Q(r)  \right) \nonumber \\
&=& 
\limsup_{n \to \infty} Q(r) n_{1}^{I}(n)^{r}  
\sum_{a \in J_{n, I}} \diam( q_{n}^{-1}(a)  )^{1+r} h(a)   \nonumber \\
& \leq &
\limsup_{n \to \infty}  Q(r) n_{1}^{I}(n)^{r} 
\sum_{a \in J_{n, I}} 
\left(  \int_{q_{n}^{-1}(a)} h^{1/(1+r)} d \lambda + \varepsilon_{3}(n) \right)^{1+r}.
\label{n_limss}
\end{eqnarray}
By completely similar arguments we obtain
\begin{eqnarray}
&& \liminf_{n \to \infty} n_{1}^{I}(n)^{r} 
\sum_{a \in J_{n, I}} \int_{q_{n}^{-1}(a)} | x - a |^{r} d \mu (x)  \nonumber \\
& \geq &
\liminf_{n \to \infty}  Q(r) n_{1}^{I}(n)^{r} 
\sum_{a \in J_{n, I}} 
\left(  \int_{q_{n}^{-1}(a)} h^{1/(1+r)} d \lambda - \varepsilon_{3}(n) \right)^{1+r}.
\label{n_limi}
\end{eqnarray}
The definitions
\[
v_{a} = \int_{q_{n}^{-1}(a)} h^{1/(1+r)} d \lambda, \qquad 
m_{0} = \min \{ v_{c} :  c \in J_{n, I}  \}
\]
imply together with (\ref{neps3}) that 
\begin{eqnarray}
1 & \leq  & \limsup_{n \to \infty}  \frac{n_{1}^{I}(n)^{r} \sum_{a \in J_{n, I}} (v_{a} + \varepsilon_{3}(n))^{1+r} }
{n_{1}^{I}(n)^{r} \sum_{a \in J_{n, I}} v_{a}^{1+r}} \nonumber \\
& \leq &
\limsup_{n \to \infty}  \frac{ (1 + \frac{\varepsilon_{3}(n)}{m_{0}})^{1+r} \sum_{a \in J_{n, I}} v_{a}^{1+r} }
{ \sum_{a \in J_{n, I}} v_{a}^{1+r}} \nonumber \\
& \leq &
\limsup_{n \to \infty}  \left(1 + \frac{\varepsilon_{3}(n)}
{  M_{1,I}^{1/(1+r)}\min \{ \diam ( q_{n}^{-1}(a)  ) : a \in J_{n, I} \}   }\right)^{1+r} = 1,
\label{1greq1}
\end{eqnarray}
with $M_{1,I}$ as defined in (\ref{m1iert}).
Similarly,
\begin{eqnarray}
1 & \geq  & \liminf_{n \to \infty}  \frac{n_{1}^{I}(n)^{r} \sum_{a \in J_{n, I}} (v_{a} - \varepsilon_{3}(n))^{1+r} }
{n_{1}^{I}(n)^{r} \sum_{a \in J_{n, I}} v_{a}^{1+r}} \nonumber \\
& \geq &
\liminf_{n \to \infty}  \left(1 - \frac{\varepsilon_{3}(n)}
{  M_{1,I}^{1/(1+r)}\min \{ \diam ( q_{n}^{-1}(a)  ) : a \in J_{n, I} \}   }\right)^{1+r} = 1.
\label{1loeq1}
\end{eqnarray}
Thus, (\ref{n_limss}) turns together with (\ref{1greq1}) and (\ref{equ_neps01}) into
\begin{eqnarray}
&& \limsup_{n \to \infty} n_{1}^{I}(n)^{r} \sum_{a \in J_{n, I}} 
\int_{q_{n}^{-1}(a)} | x - a |^{r} d \mu (x) \nonumber \\
& \leq &
\limsup_{n \to \infty}  Q(r) n_{1}^{I}(n)^{r} 
\sum_{a \in J_{n, I}} 
\left(  \int_{q_{n}^{-1}(a)} h^{1/(1+r)} d \lambda  \right)^{1+r}
\nonumber \\
& \leq &
\limsup_{n \to \infty}  Q(r) n_{1}^{I}(n)^{r} 
\sum_{a \in J_{n, I}} 
\left(  n_{1}^{I}(n)^{-1} + 
\varepsilon_{1}(n) \right)^{1+r} \left(  \int_{T_{n,I}} h^{1/(1+r)} d \lambda  \right)^{1+r} \nonumber \\
&=& Q(r) \limsup_{n \to \infty } \left( 1 + 
\varepsilon_{1}(n) n_{1}^{I}(n)  \right)^{1+r} 
\left(  \int_{T_{n,I}} h^{1/(1+r)} d \lambda  \right)^{1+r}
\nonumber \\
&=& Q(r) \left( \int_{I} h^{1/(1+r)} d \lambda \right)^{1+r}  . 
\label{lims_upp_b} 
\end{eqnarray}
On the other hand, from (\ref{n_limi}) we deduce with (\ref{1loeq1}) and (\ref{equ_neps01}) that  
\begin{eqnarray}
&& \liminf_{n \to \infty} n_{1}^{I}(n)^{r} \sum_{a \in J_{n, I}} 
\int_{q_{n}^{-1}(a)} | x - a |^{r} d \mu (x) \nonumber \\
& \geq &
\liminf_{n \to \infty}  Q(r) n_{1}^{I}(n)^{r} 
\sum_{a \in J_{n, I}} 
\left(  n_{1}^{I}(n)^{-1} - \varepsilon_{1}(n) \right)^{1+r} 
\left(  \int_{T_{n,I}} h^{1/(1+r)} d \lambda  \right)^{1+r} \nonumber \\
&=& Q(r) \left( \int_{I} h^{1/(1+r)} d \lambda \right)^{1+r}  . 
\label{limi_low_b}
\end{eqnarray}
Now, (\ref{lims_upp_b}) and (\ref{limi_low_b}) yield the assertion.
\end{proof}

Finally, we are able to state and prove the main result of the whole paper.
In the proof we rely on Lemma \ref{lemm_i1}. To this end we have to control 
$n/n_{1}^{I_{0}}(n)$ for a (suitable chosen) fixed interval $I_{0}$.
This is possible by the monotonicity of $n_{2}^{I_{0}}(\cdot)$ according to
Lemma \ref{mom_mon} (a).

\begin{theorem}
\label{main_theo}
Let $\mu$ be weakly unimodal and $(q_{n})_{n \in \mathbb{N}}$ the sequence of 
$n-$level Gersho quantizers for $\mu$. If $C_{0} < \infty$, then
$(q_{n})_{n \in \mathbb{N}}$ is asymptotically optimal.
\end{theorem}
\begin{proof}
Let $n \in \mathbb{N}$ and 
denote by $x_{1}(n)$ the right endpoint of the leftmost codecell of $q_{n}$.
Similarly, let $x_{2}(n)$ the left endpoint of the rightmost codecell of $q_{n}$.
More exactly,
\[
x_{1}(n) = \sup q_{n}^{-1}(\min q_{n} ( \mathbb{R} )), \qquad x_{2}(n) = \inf q_{n}^{-1}(\max q_{n} ( \mathbb{R} )) .
\]
Recall (cf.(\ref{equ_1_n})) that $D(\mu , q_{n} , r) \leq D(\mu , q_{1} , r)$.
Consequently, using (G4) we obtain that
\begin{eqnarray*}
&& \int_{(-\infty, x_{1}(n))} | x - \min q_{n} ( \mathbb{R} ) |^{r} d \mu (x) 
= \int_{(x_{2}(n) , \infty )} | x - \max q_{n} ( \mathbb{R} ) |^{r} d \mu (x)   \\
&=& n^{-1} D( \mu , q_{n} , r )  \leq n^{-1} D( \mu , q_{1} , r ) \to 0 \qquad \text{as } n \to \infty .
\end{eqnarray*}
Thus, Lemma \ref{lemm_ex_01} implies that
\[
x_{1}(n) \to \inf ( \supp ( \mu ) ) \quad  \text{ and } \quad  
x_{2}(n) \to \sup ( \supp ( \mu ) ) \qquad \text{ as } n \to \infty .
\] 
Now choose $n_{0}>6$ such that $\mu ( I_{0} ) > 0$ with 
\[
I_{0}=[x_{1}(n_{0}), x_{2}(n_{0})] \subset \supp ( \mu ).
\]
Now let $n \geq n_{0}$ and define $k(n) \in \{0,1,\dots \}$ such that
\begin{equation}
\label{njhu789}
n_{0} 2^{k(n)} \leq n < n_{0} 2^{k(n)+1} .
\end{equation}
Recall definition (\ref{n2in}).
By the definition of $I_{0}$ we obtain $n_{2}^{I_{0}}(n_{0} )=2$.
Due to the uniqueness of $q_{n}$ (cf. Proposition \ref{ref_prop_nempt}) we get inductively that
\[
n_{2}^{I_{0}}(n_{0} 2^{k(n)})=2^{k(n)+1} \qquad \text{resp. } \qquad n_{2}^{I_{0}}(n_{0} 2^{k(n)+1})=2^{k(n)+2} .
\]
Now, (\ref{njhu789}) and Lemma \ref{mom_mon} (a) imply
\[
2^{k(n)+1} \leq n_{2}^{I_{0}}(n) \leq 2^{k(n)+2} ,
\]
yielding
\[
\frac{n_{2}^{I_{0}}(n)}{n} \leq \frac{2^{k(n)+2}}{n_{0} 2^{k(n)}} = \frac{4}{n_{0}}
\]
Thus we deduce from definition (\ref{n2in}) and (\ref{jni_Def}) that
\begin{equation}
\label{nn12wert}
1 \leq \frac{n}{n_{1}^{I_{0}}(n)} \leq \frac{n}{n - n_{2}^{I_{0}}(n) - 2} \leq 
\frac{1}{1 - \frac{4}{n_{0}} - \frac{2}{n} } 
\leq
\frac{1}{1 - \frac{6}{n_{0}}} . 
\end{equation}
Applying (G4) we get
\begin{eqnarray}
n^{r} D(\mu , q_{n}, r) & = & 
n^{1+r}  \frac{D(\mu , q_{n}, r)}{n}   \nonumber \\
&=&
\left( \frac{n}{n_{1}^{I_{0}}} \right)^{1+r}   (n_{1}^{I_{0}})^{r}  
\sum_{a \in J_{n, I_{0}}} \int_{q_{n}^{-1}(a)} | x - a |^{r} d \mu (x) . 
\label{narr_qu_fis}
\end{eqnarray}
Now, (\ref{nn12wert}), (\ref{narr_qu_fis}) and Lemma \ref{lemm_i1} are yielding
\begin{equation}
\label{lkpoiu}
\limsup_{n \to \infty } n^{r} D ( \mu , q_{n}, r ) \leq  
\left( \frac{1}{1 - \frac{6}{n_{0}}} \right)^{1+r} Q(r) \left( \int_{I_{0}} h^{1/(1+r)} d \lambda \right)^{1+r} .
\end{equation}
By the same arguments we also deduce
\begin{equation}
\label{lboijku}
\liminf_{n \to \infty } n^{r} D ( \mu , q_{n}, r ) \geq  
Q(r) \left( \int_{I_{0}} h^{1/(1+r)} d \lambda \right)^{1+r} .
\end{equation}
Because the choice of $n_{0}>6$ was arbitrary, the assertion follows from (\ref{lkpoiu}) and (\ref{lboijku}).
\end{proof}
\begin{remark}
Alternatively, relation (\ref{lboijku}) follows also directly from Theorem \ref{stdres}.
\end{remark}

\section{Concluding remarks}

\begin{remark}
Although condition (\ref{xrdetz}) in 
Theorem \ref{stdres} cannot be dropped completely (cf. \cite[Example 6.4]{GrLu00}), it
can be weakened into $C_{0}< \infty$ for one-dimensional 
weakly unimodal distributions in view of Theorem \ref{main_theo}.
Indeed, (\ref{xrdetz}) is only needed to prove in 
Theorem \ref{stdres} that $n^{r} D_{n,r}(\mu )$ is upper-bounded by $C_{0}$, 
but this follows from Theorem \ref{main_theo} if $\mu$ is weakly unimodal and $C_{0} < \infty$.  
\end{remark}

\begin{remark}
It seems that Theorem \ref{main_theo} is also valid if the density $h$ consists of a 
finite linear combination of weakly unimodal densities.
Nevertheless it remains open, if Gersho quantizers are asymptotically optimal
for arbitrary densities with finite $r-$th moment and $C_{0} < \infty$.
Moreover, it would be interesting to know if Theorem \ref{main_theo} remains valid
for those distributions whose Gersho quantizers are unique.
\end{remark}

\begin{remark}
The methods of this paper are strictly confined to the one-dimensional case.
To find answers regarding the existence, construction and asymptotic optimality of Gersho quantizers in 
higher dimensions further research is necessary.
\end{remark}

\begin{remark}
For dyadic homogeneous one-dimensional Cantor distributions 
\cite{GrLus97, KeZh07, Kr08} it is known that
\[
0 < \underline{C} = 
\liminf_{n \to \infty } n^{r} D_{n,r}(\mu ) < 
\limsup_{n \to \infty } n^{r} D_{n,r}(\mu ) = \overline{C} < \infty .
\]
It is easy to check for these measures that for $n=2^{k}$ a unique optimal quantizer $q_{k}$ exists,
these quantizers are also Gersho quantizers, and that $2^{kr} D_{2^{k},r}(\mu )$
tends to $\underline{C}$. Hence,
\begin{equation}
\label{underex}
\underline{C} = \liminf_{n \to \infty } n^{r} \inf \{ D(\mu ,q, r ) : q \in \mathcal{G}(n,\mu, r) \}
\end{equation}
and by definition
\begin{equation}
\label{overex}
\overline{C} \leq \limsup_{n \to \infty } n^{r} \inf \{ D(\mu ,q, r ) : q \in \mathcal{G}(n,\mu, r) \}.
\end{equation}
It remains open, if (\ref{overex}) is also an equation or not.
Moreover, there exist distributions \cite{GrLu05} which are non-atomic and singular to the Lebesgue measure but
$\underline{C} = \overline{C}$. It remains also open, if (\ref{underex}) is also valid for 
these distributions or 
if even (\ref{overex}) turns into an equation for such measures.
\end{remark}

\begin{remark}
Let $I \subset \mathbb{R}$ be a compact interval and let 
$q_{n}$ be a quantizer with $n$ codepoints. Recall definitions (\ref{jni_Def}).
If $(q_{n})_{n \in \mathbb{N}}$ is asymptotically optimal for $\mu = h \lambda$ and (\ref{xrdetz}) is satisfied, 
then Bucklew \cite{Bu84} has shown that  
\begin{equation}
\label{dens_conv}
\frac{n_{1}^{I}(n)}{n} \to \frac{\int_{I}h^{\frac{1}{1+r}}d \lambda}{\int_{\mathbb{R}}h^{\frac{1}{1+r}}d \lambda}
\qquad \text{as } n \to \infty .
\end{equation}
Hence, Remark (\ref{implfin}) and Theorem \ref{main_theo} implies that (\ref{dens_conv}) holds also 
for Gersho quantizers if $\mu$ is weakly unimodal and satisfies (\ref{xrdetz}).
Beside of this point density result, one is also interested in densities for the error and the mass of the codecells. 
More formally, let 
\[
E(n,I,q_{n}) = \frac{\int_{I} | x - q_{n}(x) |^{r} d \mu (x) }{D(\mu , q_{n}, r)} 
\]
and
\[
M(n, I, q_{n}) = 
\sup \{  | n \mu ( q_{n}^{-1}(a) ) - h^{r/(1+r)}(a) \cdot \left( \int h^{1/(1+r)} d \lambda \right)  |   : 
a \in J_{n,I}  \} .
\]
By analyzing the approach in \cite{Bu84} and \cite{KrLi11}, it is reasonable to conjecture 
that 
\begin{equation}
\label{asymp_err}
E(n,I,q_{n}) \to \frac{\int_{I}h^{\frac{1}{1+r}}d \lambda}{\int_{\mathbb{R}}h^{\frac{1}{1+r}}d \lambda} 
\qquad \text{as } n \to \infty ,
\end{equation}
i.e. that error and point density are asymptotically equal for asymptotically optimal quantizers, which would imply
that (\ref{asymp_err}) holds also for Gersho quantizers if $\mu$ is weakly unimodal and satisfies (\ref{xrdetz}).
Finally, due to \cite[Theorem 4]{DeFoPa04} we know that 
$M(n, I, q_{n}) \to 0$ if $h$ is weakly unimodal and $q_{n}$ is an
asymptotically optimal stationary Voronoi quantizer. 
Note that a Voronoi quantizer is stationary if every codecell has non-vanishing $\mu-$mass and (G3) is satisfied.
Although, Gersho quantizers satisfy (G3) by definition they need not to be Voronoi quantizers.
However, we
conjecture that $M(n, I, q_{n}) \to 0$ still holds for Gersho quantizers if $h$ is weakly unimodal and
(\ref{xrdetz}) is satisfied.  
\end{remark}

\begin{remark}
As already mentioned in the introduction, Delattre et al. \cite{DeFoPa04} have shown that
asymptotically optimal stationary quantizers $q_{n}$ satisfy (G4) asymptotically, i.e. 
\begin{equation}
\label{g4prop}
G(n, I, q_{n}) = \sup \{  | n^{1+r} \int_{q_{n}^{-1}(a)} | x - a |^{r} d \mu (x) - C_{0} |   : a \in J_{n,I}  \} \to 0 
\end{equation}
as $n \to \infty$.
Note, that asymptotic optimality does not imply (\ref{g4prop}).
To see this let $\mu$ be the uniform distribution on $[0,1]$, let $\varepsilon \in (0,1)$ and consider the quantizer
\begin{eqnarray*}
&& q_{n}^{\varepsilon }(x) \\
&=& \frac{\varepsilon}{n} \cdot 1_{[0, \frac{\varepsilon}{n} )}(x)  + \sum_{i=1}^{n-1} 
\left( \frac{\varepsilon}{n} + \frac{1 - \frac{\varepsilon}{n}}{n-1} \left( \frac{3}{2}i - 1 \right)  \right) \cdot
1_{[\frac{\varepsilon}{n} + (i-1) \cdot \frac{1 - \frac{\varepsilon}{n}}{n-1} , \frac{\varepsilon}{n} 
+ i \cdot \frac{1 - \frac{\varepsilon}{n}}{n-1} )}(x).
\end{eqnarray*} 
By a straightforward calculation one gets 
\[
n^{r} D( \mu , q_{n}^{\varepsilon } , r ) = Q(r) \left( \varepsilon^{1+r} n^{-1} + 
\frac{n-1}{n} \left( 1 + \frac{1 - \varepsilon }{n-1} \right)^{1+r}   \right) \to Q(r)
\]
as $n \to \infty$. Hence, $(q_{n}^{\varepsilon }(x))_{n \in \mathbb{N}}$ is asymptotically optimal
and satisfies (G1), (G2) and (G3), but satisfies neither (G4) nor (\ref{g4prop}). 

Assume that (\ref{xrdetz}) is satisfied and the distribution is weakly unimodal.
By Theorem \ref{stdres} and Theorem \ref{main_theo} we know that optimal quantizers and 
Gersho quantizers are asymptotically optimal. Nevertheless, quantizers can exist - as shown by this example -
which are neither optimal nor Gersho quantizers but asymptotically optimal.
Motivated by this and (\ref{g4prop}) let us define the condition 
\[
(G5) n^{1+r} \max \{ | \int_{q_{n}^{-1}(a)} |x - a|^{r} d \mu (x) - \frac{D(\mu , q_{n} , r)}{n} | 
: a \in q_{n}(\mathbb{R})  \} \to 0 \text{ as } n \to \infty .
\]
Now consider the class of quantizers which are satisfying conditions (G1), (G2), (G3) and (G5).
Of course, (G5) is weaker than (G4). Hence, Gersho quantizers are contained in this class but it is
open if optimal quantizers are also a part of this class.
Again, the example from above demonstrates that this class does not contain all asymptotically optimal
quantizers. It remains open if every quantizer of this class is asymptotically optimal. 
\end{remark}

\begin{remark}
To the authors knowledge there are no general results available so far 
concerning the rate of convergence in (\ref{asympform}).
Fort and Pages \cite[Theorem 5]{FoPa02} have shown for three special families of
scalar distributions (exponential, power and inverse power) that
\begin{equation}
\label{asymprate}
\limsup_{n \to \infty } |C_{0} - n^{r}D( \mu , q_{n} , r )| \frac{n}{\log (n)} < \infty ,
\end{equation}
if $q_{n} \in \mathcal{C}(n , \mu , r)$.
It remains open, if (\ref{asymprate}) is still true for Gersho quantizers.
Moreover, it is unclear if (\ref{asymprate}) is generally true for weakly unimodal densities
and Gersho quantizers.
\end{remark}

\subsubsection*{Acknowledgements}
The author is indebted to the referee for his/her valuable 
suggestions and comments, which improved the quality and presentation of this paper.

\noindent 
Wolfgang Kreitmeier \\
Department of Informatics and Mathematics \\ 
University of Passau \\
Innstra\ss e 33, 94032 Passau, Germany \\
E-mail: \texttt{wkreitmeier@gmx.de} \\
Phone:  \texttt{+49851/509-3014} \\
Fax: \texttt{+49851/509-373014} \bigskip \\

\end{document}